\documentclass[a4paper,12pt,intlimits,oneside,reqno]{amsart}
\usepackage{enumerate}
\usepackage{amsfonts,amsmath, amsxtra,amssymb,latexsym, amscd,amsthm}

\usepackage{latexsym,amssymb}
\newtheorem{theorem}{Theorem}[section]
\newtheorem{lemma}[theorem]{Lemma}

\theoremstyle{corollary}
\newtheorem{corollary}[theorem]{Corollary}

\theoremstyle{definition}
\newtheorem{definition}[theorem]{Definition}

\theoremstyle{remark}

\numberwithin{equation}{section}

\newcommand{\comment}[1]{}

\begin{document}

\title [Multilinear Hausdorff operator]{Multilinear Hausdorff operators on some function spaces with variable exponent}

\author{Nguyen Minh Chuong}

\address{Institute of mathematics, Vietnamese  Academy of Science and Technology,  Hanoi, Vietnam.}
\email{nmchuong@math.ac.vn}

\author{Dao Van Duong}
\address{Shool of Mathematics,  Mientrung University of Civil Engineering, Phu Yen, Vietnam.}
\email{daovanduong@muce.edu.vn}

\author{Kieu Huu Dung}
\address{Shool of Mathematics, University of Transport and Communications- Campus in Ho Chi Minh City, Vietnam.}
\email{khdung@utc2.edu.vn}
\begin{abstract}
The aim of the present paper is to give necessary and sufficient conditions for the boundedness of a general class of multilinear Hausdorff operators that acts on the product of some weighted  function spaces with variable exponent such as the weighted Lebesgue, Herz, central Morrey and Morrey-Herz type spaces with variable exponent.  Our results improve and generalize some previous known results.
\end{abstract}

\maketitle

\section{Introduction}\label{section1}
The one dimensional Hausdorff operator is defined by
$$ H_{\Phi}(f)(x)=\int\limits_{0}^\infty{\frac{\Phi(y)}{y}f\left(\frac{x}{y}\right)dy},$$ 
where $\Phi$ is an integrable function on the positive half-line. The Hausdorff operator may be originated by Hurwitz and Silverman \cite{Hurwitz} in order to study summability of number series (see also \cite{Hausdorff}). It is well known that the Hausdorff operator is one of important operators in harmonic analysis, and it is used to solve certain classical problems in analysis. It is worth pointing out that if the kernel function $\Phi$ is taken appropriately, then  the Hausdorff operator reduces to many classcial operators in analysis such as the Hardy operator, the Ces\`{a}ro operator, the Riemann-Liouville fractional integral operator and the Hardy-Littlewood average operator (see, e.g., \cite{Andersen}, \cite{Christ}, \cite{FGLY2015}, \cite{Miyachi} and references therein).
\vskip 5pt
In 2002, Brown and M\'{o}ricz \cite{BM} extended the study of  Hausdorff operator to the high dimensional space.
Given $\Phi$ be a locally integrable function on $\mathbb R^n$, the $n$-dimensional Hausdorff operator $H_{\Phi,A}$ associated to the kernel function $\Phi$ is then defined in terms of the integral form as follows
\begin{equation}\label{Hausdorff1}
H_{\Phi, A}(f)(x)=\int\limits_{\mathbb R^n}{\frac{\Phi(t)}{|t|^n}f(A(t) x)dt},\,x\in\mathbb R^n,
\end{equation}
where $A(t)$ is an $n\times n$ invertible matrix for almost everywhere $t$ in the support of $\Phi$.  It should be pointed out that if we take  $\Phi(t)=|t|^n\psi(t_1)\chi_{[0,1]^n}(t)$ and $A(t)= t_1.I_n$ ($I_n$ is an identity matrix), for $t=(t_1,t_2,...,t_n)$, where $\psi:[0,1]\to [0,\infty)$ is a measurable function,  $H_{\Phi,A}$ then reduces to the weighted Hardy-Littlewood average operator due to Carton-Lebrun and Fosset \cite{Carton-Lebrun} defined as the following
\begin{equation}\label{Upsi}
U_{\psi}(f)(x)=\int\limits_{0}^{1}{f(tx)\psi(t)dt},\,\,x\in\mathbb R^n.
\end{equation}
Similarly, by taking $\Phi(t)=|t|^n\psi(t_1)\chi_{[0,1]^n}(t)$ and $A(t)= s(t_1).I_n$, with $s:[0,1]\to \mathbb R$ being a measurable function, it is easy to see that $H_{\Phi,A}$ reduces to the weighted Hardy-Ces\`{a}ro operator $U_{\psi,s}$  defined by  Chuong and Hung \cite{CH2014} as follows
\begin{equation}\label{Upsi}
U_{\psi,s}(f)(x)=\int\limits_{0}^{1}{f(s(t)x)\psi(t)dt},\,\,x\in\mathbb R^n.
\end{equation}
\vskip 5pt
In recent years, the theory of weighted Hardy-Littlewood average operators, Hardy-Ces\`{a}ro operators and  Hausdorff operators  have been significantly developed into different contexts (for more details see \cite{BM}, \cite{CH2014}, \cite{CDH2016}, \cite{CHH2016}, \cite{Moricz2005}, \cite{Xiao2001} and references therein). Also, remark that Coifman and Meyer \cite{Coifman1}, \cite{Coifman2} discovered  a multilinear point of view in their study of certain singular integral operators. Thus, the research of the theory of multilinear operators is not only attracted by a pure question to generalize the theory of linear ones but also by their deep applications in harmonic analysis. In 2015, Hung and Ky \cite{HK2015} introduced the weighted multilinear Hardy-Ces\`{a}ro type operators, which are generalized of weighted multilinear Hardy operators in \cite{FGLY2015}, defined as follows: 
\begin{equation}\label{MulUpsi}
U_{\psi,\vec {s}}^{m, n}(f_1,...,f_m)(x)=\int\limits_{[0,1]^n}{\Big(\prod\limits_{i=1}^{m}f_i(s_i(t)x)\Big)\psi(t)dt},\,\,x\in\mathbb R^n,
\end{equation}
where  $\psi:[0,1]^n\to [0,\infty)$, and $s_1,...,s_m:[0,1]^n\to \mathbb R$  are measurable functions. By the relation between Hausdorff operator and Hardy-Ces\`{a}ro operator as mentioned above, we shall introduce in this paper a more general multilinear operator of Hausdorff type defined as follows.
\begin{definition}
Let $\Phi:\mathbb R^n \to [0,\infty)$ and $A_i(t)$, for $i=1,...,m$, be $n\times n$  invertible matrices for almost everywhere $t$ in the support of $\Phi$. Given $f_1, f_2, ..., f_m:\mathbb R^n\to \mathbb C$ be measurable functions, the multilinear Hausdorff operator $H_{\Phi,\vec A}$ is defined  by
\begin{equation}\label{mulHausdorff}
{H_{\Phi ,\vec{A} }}(\vec{f})(x) = \int\limits_{{\mathbb R^n}} {\frac{{\Phi (t)}}{{{{\left| t \right|}^n}}}} \prod\limits_{i = 1}^m {{f_i}} ({A_i}(t)x)dt,\,x\in\mathbb R^n,
\end{equation}
for $\vec{f}=\left(f_1, ..., f_m\right)$ and $\vec{A}=\left(A_1, ..., A_m\right)$.
\end{definition}
It is obvious that when $\Phi(t)=|t|^n.\psi(t)\chi_{[0,1]^n}(t)$ and $A_i(t)= s_i(t).I_n$, where $\psi:[0,1]^n\to [0,\infty), s_1,...,s_m:[0,1]^n\to \mathbb R$ are  measurable functions, then the multilinear Hausdorff operator $H_{\Phi,\vec A}$ reduces to the weighted multilinear Hardy-Ces\`{a}ro operator $U_{\psi,\vec {s}}^{m, n}$ above.
\vskip 5pt
 It is also interesting that the theory of function spaces with variable exponents has attracted much more attention because of the necessary  in the field of electronic fluid mechanics and its important applications to the elasticity, fluid dynamics, recovery of graphics, and differential equations (see \cite{Almeida}, \cite{Chuong},  \cite{Chuong2016}, \cite{Diening}, \cite{H2000},  \cite{Jacob},  \cite{CUF2013}). The foundational results and powerful applications of  some function spaces with variable exponents in harmonic analysis and partial differential equations are given in the books \cite{CUF2013}, \cite{Diening1} and the references therein. It is well-known that the Calder\'{o}n-Zygmund singular operators, the Hardy-type operators and their commutators have been extensively investigated on  the Lebesgue, Herz, Morrey, and  Morrey-Herz spaces with variable exponent (see, e.g., \cite{Bandaliev2010}, \cite{Capone}, \cite{Cruz-Uribe}, \cite{Guliyev}, \cite{Mamedov}, \cite{Mashiyev},  \cite{LZ2014}, \cite{Rafeiro}, \cite{WZ2016}, and others).
\vskip 5pt
Motivated by above mentioned results,  the goal of this paper is to establish the necessary and sufficient conditions for the boundedness of multilinear Hausdorff operators on the product of weighted Lebesgue, central Morrey, Herz, and Morrey-Herz spaces with variable exponent.  In each case, the estimates for operator norms are worked out. It should be pointed out that all results in this paper are new even in the case of linear Hausdorff operators.
\vskip 5pt
Our paper is organized as follows. In Section 2, we give necessary preliminaries on weighted Lebesgue spaces, central Morrey spaces, Herz spaces and Morrey-Herz spaces with variable exponent. Our main theorems are given and proved in Section 3.
\section{Preliminaries}\label{section2}
Before stating our results in the next section, let us give some basic facts and notations which will be used throughout this paper. By $\|T\|_{X\to Y}$, we denote the norm of $T$ between two normed vector spaces $X$ and $Y$. The letter $C$ denotes a positive constant which is independent of the main parameters, but may be different from line to line. Given a measurable set $\Omega$, let us denote by $\chi_\Omega$ its characteristic function, by $|\Omega|$ its Lebesgue measure, and by $\omega(\Omega)$  the integral $\int\limits_{\Omega}\omega(x)dx$. For any $a\in\mathbb R^n$ and $r>0$, we denote by $B(a,r)$ the ball centered at $a$ with radius $r$.
\vskip 5pt
Next, we write $a \lesssim b$ to mean that there is a positive constant $C$, independent of the main parameters, such that $a\leq Cb$. The symbol $f\simeq g$ means that $f$ is equivalent to $g$ (i.e.~$C^{-1}f\leq g\leq Cf$). 

In what follows, we denote $\chi_k=\chi_{C_k}$, $C_k=B_k\setminus B_{k-1}$ and $B_k = \big\{x\in \mathbb R^n: |x| \leq 2^k\big\}$, for all $k\in\mathbb Z$. Now, we are in a position to give some notations and definitions of Lebesgue, Herz, Morrey and Morrey-Herz spaces with constant parameters. For further information on these spaces as well as their deep applications in analysis, the interested readers may refer to the work \cite{ALP} and to the monograph \cite{Lu}.

In this paper, as usual, we will denote by $\omega(\cdot)$ a non-negative weighted function on $\mathbb R^n$.
\begin{definition} Let $1\leq p<\infty$, we define the weighted Lebesgue space $L^p(\omega)$ of a measurable function $f$ by
\[
\big\|f\big\|_{L^p(\omega)} = \left(\int\limits_{\mathbb R^n} {{{\left| {f(x)} \right|^p\omega(x)}}dx}  \right)^{\frac{1}{p}}<\infty,
\]
and for $p=\infty$ by
\[
\big\|f\big\|_{L^\infty(\omega)} = {\rm inf} \big\{M>0: \omega\big(\big\{x\in \mathbb R^n:|f(x)|>M\big\}\big)=0 \big\}<\infty.
\]
\end{definition}
\begin{definition}
Let $\alpha\in\mathbb R$, $0<q<\infty$, and $0<p<\infty$. The weighted homogeneous Herz-type space ${\mathop{K}\limits^.}^{\alpha, p}_q(\omega)$ is defined by
\[
{\mathop{K}\limits^.}^{\alpha, p}_q(\omega)=\left\{f\in L^q_\text {loc}(\mathbb R^n\setminus\{0\},\omega):\|f\|_{{\mathop{K}\limits^.}^{\alpha,p}_q(\omega)} <\infty \right \},
\]
where $\|f\|_{{\mathop{K}\limits^.}^{\alpha,p}_q(\omega)}=\left(\sum\limits_{k=-\infty}^\infty 2^{k\alpha p}\|f\chi_k\|^p_{L^q(\omega)}\right)^{\frac{1}{p}}.$
\end{definition}

\begin{definition}
Let $\lambda \in \mathbb R$ and $1\leq p<\infty$. The weighted central Morrey space ${\mathop B\limits^.}^{p,\lambda}({\omega})$
is defined by the set of all locally $p$-integrable functions $f$ satisfying
\[
{\left\| f \right\|_{{\mathop B\limits^.}^{p,\lambda}({\omega})}} = \mathop {\sup }\limits_{R\, > 0} {\left( {\dfrac{1}{{{\omega}(B(0,R))^{1+\lambda p}}}\int\limits_{B(0,R)} {{{\left| {f(x)} \right|}^p}\omega(x)dx} } \right)^{\frac{1}{p}}} < \infty .
\]
\end{definition}
\begin{definition}
Let $\alpha \in \mathbb R$, $0 < p < \infty, 0 < q <\infty, \lambda \geq 0$. The homogeneous weighted Morrey-Herz type space ${M\mathop{K}\limits^.}^{\alpha,\lambda}_{p,q}(\omega)$ is defined by
\[
{M\mathop{K}\limits^.}^{\alpha,\lambda}_{p,q}(\omega)=\left\{f\in L^q_\text {loc}(\mathbb R^n\setminus\{0\},\omega):\|f\|_{{M\mathop{K}\limits^.}^{\alpha,\lambda}_{p,q}(\omega)} <\infty \right \},
\]
where $\|f\|_{{M\mathop{K}\limits^.}^{\alpha,\lambda}_{p,q}(\omega)}=\sup\limits_{k_0\in\mathbb Z}2^{-k_0\lambda}\left(\sum\limits_{k=-\infty}^{k_0} 2^{k\alpha p} \|f\chi_k\|^p_{L^q(\omega)}\right)^{\frac{1}{p}}.$
\end{definition}
\hspace {-15pt} {\bfseries Remark 1.} It is useful to note that ${\mathop{K}\limits^.}^{0,p}_{p}(\mathbb R^n)= L^p(\mathbb R^n)$ for $0<p<\infty$; ${\mathop K\limits^.}^{\alpha/p,p}_{p}(\mathbb R^n)= L^p(|x|^\alpha dx)$ for all $0<p<\infty$ and $\alpha\in\mathbb R$. Since ${M\mathop{K}\limits^.}^{\alpha,0}_{p,q}(\mathbb R^n)$ =${\mathop{K}\limits^.}^{\alpha, p}_{q}(\mathbb R^n)$, it follows that the Herz space is a special case of Morrey-Herz space. Therefore, the Herz spaces are natural generalisations of the Lebesgue spaces with power weights.
\vskip 5pt
Now, we present the definition of the Lebesgue  space with variable exponent. For further readings on its deep applications in harmonic analysis, the interested reader may find in the works \cite{CUF2013}, \cite{Diening} and  \cite{Diening1}.
\begin{definition}
Let $\mathcal P(\mathbb R^n)$ be the set of all measurable functions $p(\cdot):\mathbb R\to [1,\infty]$. For $p(\cdot)\in \mathcal P(\mathbb R^n)$, the variable exponent Lebesgue space $L^{p(\cdot)}(\mathbb R^n)$ is the set of all complex-valued measurable functions $f$ defined on $\mathbb R^n$ such that there exists constant $\eta >0$ satisfying
\[
F_{p}(f/\eta):=\int\limits_{\mathbb R^n\setminus\Omega_\infty}\left({\frac{|f(x)|}{\eta}}\right)^{p(x)}dx +\big\|f/\eta\big\|_{L^{\infty}(\Omega_\infty)}<\infty,
\]
where $\Omega_{\infty}=\big\{ x\in\mathbb R^n: p(x)=\infty\big\}$. When $|\Omega_{\infty}|=0$, it is straightforward
\[
F_{p}(f/\eta)=\int\limits_{\mathbb R^n}\left({\frac{|f(x)|}{\eta}}\right)^{p(x)}dx<\infty.
\]
\end{definition}
The variable exponent Lebesgue space $L^{p(\cdot)}(\mathbb R^n)$ then becomes a norm space equipped with a norm given by
$$\left\|f\right\|_{L^{p(\cdot)}}= \inf \left\{\eta>0: F_p\left(\frac{f}{\eta}\right)\leq 1\right\}.$$
Let us denote by $\mathcal P_{b}(\mathbb R^n)$ the class of exponents $q(\cdot)\in\mathcal P(\mathbb R^n)$ such that
\[
1 < q_{-}\leq q(x) \leq q_{+}<\infty,\,  \text{ for all }\,x\in\mathbb R^n,
\]
where $q_{-}= \text{ess\,inf}_{x\in\mathbb R^n}q(x)$ and $q_{+}= \text{ess\,sup}_{x\in\mathbb R^n}q(x)$.
For $p\in\mathcal P_{b}(\mathbb R^n)$, it is useful to remark that we have the following inequalities which are usually used in the sequel.
\begin{eqnarray}\label{maxminvar}
&&\hskip -30pt [i]\,\, \textit{ \rm If }\, F_{p}(f)\leq C, \textit{ \rm then}\, \big\|f\big\|_{L^{p(\cdot)}} \leq \textit{\rm max}\big\{ {C}^{\frac{1}{q_-}}, {C}^{\frac{1}{q_+}} \big\},\,\textit{\rm for all}\,\, f\in L^{p(\cdot)}(\mathbb R^n),
\nonumber
\\
&&\hskip - 30pt [ii]\textit{ \rm If }\, F_{p}(f)\geq C, \textit{ \rm then}\, \big\|f\big\|_{L^{p(\cdot)}}\geq \textit{\rm min}\big\{ {C}^{\frac{1}{q_-}}, {C}^{\frac{1}{q_+}} \big\},\,\textit{\rm for all}\, f\in L^{p(\cdot)}(\mathbb R^n).
\end{eqnarray}
\\
The space ${\mathcal P}_\infty(\mathbb R^n)$ is defined by the set of all measurable functions $q(\cdot)\in\mathcal P(\mathbb R^n)$ and there exists a constant $q_\infty$ such that
\[
q_{\infty} =\lim\limits_{|x|\to\infty}{q(x)}.
\]
\\
For $p(\cdot)\in\mathcal P(\mathbb R^n)$, the weighted variable exponent Lebesgue space $L^{p(\cdot)}_\omega(\mathbb R^n)$ is the set of all complex-valued measurable functions $f$ such that $f\omega$ belongs to the $L^{p(\cdot)}(\mathbb R^n)$ space, and the norm of $f$ in  $L^{p(\cdot)}_\omega(\mathbb R^n)$ is given by
\[
\big\|f\big\|_{L^{p(\cdot)}_\omega}=\big\|f\omega\big\|_{L^{p(\cdot)}}.
\]
Let $\mathbf C_0^{\text{log}}(\mathbb R^n) $ denote the set of  all log-H\"{o}lder continuous functions $\alpha (\cdot)$ satisfying at the origin
\[
\left|\alpha(x)-\alpha(0)\right|\leq \dfrac{C_0^\alpha}{\text{log}\left(e+\frac{1}{|x|}\right)},\,  \text{ for all }\, x\in\mathbb R^n.
\]
Denote by $\mathbf C_\infty^{\text{log}}(\mathbb R^n) $  the set of all log-H\"{o}lder continuous functions $\alpha (\cdot)$ satisfying at infinity
\[
\left|\alpha(x)-\alpha_{\infty}\right|\leq \dfrac{C_\infty^\alpha}{\text{log}(e+|x|)},\,  \text{ for all }\, x\in\mathbb R^n,
\]
\vskip 5pt
Next, we would like to give the definition of variable exponent weighted Herz spaces ${\mathop{K}\limits^.}^{\alpha(\cdot),p}_{q(\cdot),\omega}$ and the definition of variable exponent weighted Morrey-Herz spaces ${M\mathop{K}\limits^.}^{\alpha(\cdot),\lambda}_{p,q(\cdot),\omega}$ (see \cite{LZ2014}, \cite{WZ2016} for more details).
\begin{definition}
Let $0<p<\infty,  q(\cdot)\in \mathcal P_b(\mathbb R^n)$ and $\alpha(\cdot):\mathbb R^n\to \mathbb R$ with $\alpha(\cdot)\in L^{\infty}(\mathbb R^n)$. The variable exponent weighted Herz space ${\mathop{K}\limits^.}^{\alpha(\cdot),p}_{q(\cdot),\omega}$ is defined
by
\[
{\mathop{K}\limits^.}^{\alpha(\cdot),p}_{q(\cdot),\omega}=\left\{f\in L^{q(\cdot)}_\text {loc}(\mathbb R^n\setminus\{0\}):\|f\|_{{\mathop{K}\limits^.}^{\alpha(\cdot),p}_{q(\cdot),\omega}} <\infty \right \},
\]
where $\|f\|_{{\mathop{K}\limits^.}^{\alpha(\cdot),p}_{q(\cdot),\omega}}=\left(\sum\limits_{k=-\infty}^{\infty} \|2^{k\alpha(\cdot)} f\chi_k\|^p_{L^{q(\cdot)}_\omega}\right)^{\frac{1}{p}}.$
\end{definition}
\begin{definition}
Assume that $0\leq \lambda<\infty, 0<p<\infty,  q(\cdot)\in \mathcal P_b(\mathbb R^n)$ and $\alpha(\cdot):\mathbb R^n\to \mathbb R$ with $\alpha(\cdot)\in L^{\infty}(\mathbb R^n)$. The variable exponent weighted Morrey-Herz space ${M\mathop{K}\limits^.}^{\alpha(\cdot),\lambda}_{p,q(\cdot), \omega}$ is defined
by
\[
{M\mathop{K}\limits^.}^{\alpha(\cdot),\lambda}_{p,q(\cdot),\omega}=\left\{f\in L^{q(\cdot)}_\text {loc}(\mathbb R^n\setminus\{0\}):\|f\|_{{M\mathop{K}\limits^.}^{\alpha(\cdot),\lambda}_{p,q(\cdot),\omega}} <\infty \right \},
\]
where $\|f\|_{{M\mathop{K}\limits^.}^{\alpha(\cdot),\lambda}_{p,q(\cdot),\omega}}=\sup\limits_{k_0\in\mathbb Z}2^{-k_0\lambda}\left(\sum\limits_{k=-\infty}^{k_0} \|2^{k\alpha(\cdot)} f\chi_k\|^p_{L^{q(\cdot)}_\omega}\right)^{\frac{1}{p}}.$
\end{definition}
Note that, when $p(\cdot)$, $q(\cdot)$ and $\alpha(\cdot)$ are constant, it is obvious to see that
\begin{equation}\label{rel-func}
L^{p}_{{\omega}^{1/p}}= L^{p}(\omega),\,\,{\mathop{K}\limits^.}^{\alpha,p}_{q,{\omega}^{1/p}}={\mathop{K}\limits^.}^{\alpha,p}_{q}(\omega)\,\,\textit{\rm and}\,\,{{M\mathop{K}\limits^.}^{\alpha,\lambda}_{p,q,{\omega}^{1/p}}}={{M\mathop{K}\limits^.}^{\alpha,\lambda}_{p,q}(\omega)}.
\end{equation}

Because of defining of weighted Morrey-Herz spaces with variable exponent and Proposition 2.5 in \cite{LZ2014}, we have the following result. The proof is trivial and may be found in \cite{WZ2016}.
\begin{lemma}\label{lemmaVE}
Let $\alpha(\cdot)\in L^{\infty}(\mathbb R^n)$, $q(\cdot)\in \mathcal P_b(\mathbb R^n)$, $p\in(0, \infty)$ and $\lambda\in [0,\infty)$. If $\alpha(\cdot)$ is log-H\"{o}lder continuous
both at the origin and at infinity, then
\[
\big\|f\chi_j\big\|_{L^{q(\cdot)}_\omega} \leq C. 2^{j(\lambda-\alpha(0))}\big\|f\big\|_{{M\mathop{K}\limits^.}^{\alpha(\cdot),\lambda}_{p,q(\cdot),\omega}},\,\,\textit{for all}\,\, j\in\mathbb Z^-,
\]
and
\[
\big\|f\chi_j\big\|_{L^{q(\cdot)}_\omega} \leq C. 2^{j(\lambda-\alpha_\infty)}\big\|f\big\|_{{M\mathop{K}\limits^.}^{\alpha(\cdot),\lambda}_{p,q(\cdot),\omega}},\,\,\textit{for all}\,\, j\in\mathbb N.
\]
\end{lemma}
We also extend to define two-weight $\lambda$-central Morrey spaces with variable-exponent as follows.
\begin{definition}
For $\lambda \in \mathbb R$ and $p\in\mathcal P_\infty(\mathbb R^n)$, we denote ${\mathop B\limits^.}_{\omega_1,\omega_2}^{p(\cdot),\lambda}$ the class of locally integrable functions $f$ on $\mathbb R^n$ satisfying 
\[
{\big\| f \big\|_{{\mathop B\limits^.}_{\omega_1,\omega_2}^{p(\cdot),\lambda }}} = \mathop {\sup }\limits_{R\, > 0}  {\frac{1}{{{\omega_1}\big(B(0,R)\big)^{\lambda+\frac{1}{p_{\infty}}}}}} \big\|f\big\|_{L^{p(\cdot)}_{\omega_2}(B(0,R))}< \infty,
\]
where $\big\|f\big\|_{L^{p(\cdot)}_{\omega_2}(B(0,R))}=\big\|f\chi_{B(0,R)}\big\|_{L^{p(\cdot)}_{\omega_2}}$
and $\omega_1$, $\omega_2$ are non-negative and local integrable functions.  Moreover, as $p(\cdot)$ is constant and $\omega_1=\omega$ and $\omega_2=\omega^{\frac{1}{p}}$, it is natural to see that  ${\mathop B\limits^.}_{\omega,\omega^{1/p}}^{p(\cdot),\lambda}={\mathop B\limits^.}^{p,\lambda}({\omega})$.
\end{definition}
Later, the next theorem is stated as an embedding result for the Lebesgue spaces with variable exponent (see, for example, Theorem 2 in \cite{Bandaliev2010}, Theorem 2.45 in \cite{CUF2013}, Lemma 3.3.1 in \cite{Diening1}).
\begin{theorem}\label{theoembed}
Let $p(\cdot), q(\cdot)\in\mathcal P(\mathbb R^n)$ and $q(x)\leq p(x)$ almost everywhere $x\in \mathbb R^n$, and
\[
\frac{1}{r(\cdot)}:=\frac{1}{q(\cdot)}-\frac{1}{p(\cdot)}\,\,\textit{\rm and}\,\, \big\|1\big\|_{L^{r(\cdot)}}<\infty.
\]
Then there exists a constant K such that
\[
\big\|f\big\|_{L^{q(\cdot)}_\omega}\leq K \big\|1\big\|_{L^{r(\cdot)}}\big\|f\big\|_{L^{p(\cdot)}_\omega}.
\]
\end{theorem}
\section{Main results and their proofs}
Before stating the next main results, we introduce some notations which will be used throughout this section. Let $ \gamma_1,...,\gamma_m\in\mathbb R,\lambda_1,...,\lambda_m \geq 0, p, p_i\in (0,\infty)$, $q_i\in\mathcal P_b(\mathbb R^n)$ for $i=1,...,m$ and $\alpha_1,...,\alpha_m\in L^\infty(\mathbb R^n)\cap \mathbf C_0^{\text{log}}(\mathbb R^n)\cap \mathbf C_{\infty}^{\text{log}}(\mathbb R^n)$. The functions $\alpha(\cdot),q(\cdot)$ and numbers $\gamma,\lambda$ are defined by
$$   {\alpha _1(\cdot)} + \cdots + {\alpha _m(\cdot)} = \alpha(\cdot),  $$
  $$\frac{1}{{{q_1(\cdot)}}}+ \cdots  + \frac{1}{{{q_m(\cdot)}}} = \frac{1}{q(\cdot)}, $$
  $$\gamma_1+\cdots+\gamma_m=\gamma,$$
$$  {\lambda _1} + {\lambda _2} + \cdots + {\lambda _m} = \lambda. $$
Thus, it is clear to see that the function $\alpha$ also belongs to the $L^\infty(\mathbb R^n)\cap \mathbf C_0^{\text{log}}(\mathbb R^n)\cap \mathbf C_{\infty}^{\text{log}}(\mathbb R^n)$ space.
\vskip 5pt
For a matrix $A=(a_{ij})_{n\times n}$, we define the norm of $A$ as follow
\begin{equation}\label{normB}
\left\|A\right\|=\left(\sum\limits_{i,j=1}^{n}{{|a_{ij}|}^2}\right)^{1/2}.
\end{equation}
As above we conclude $\left|Ax\right|\leq \left\|A\right\|\left|x\right|$ for any vector $x\in\mathbb R^n$. In particular, if $A$ is invertible,  then we find
\begin{equation}\label{detA}
\left\|A\right\|^{-n}\leq \left|\rm det(A^{-1})\right|\leq \left\|A^{-1}\right\|^{n}.
\end{equation}
In this section, we will  investigate the boundedness of multilinear Hausdorff operators on variable exponent Herz, central Morrey and Morrey-Herz spaces with power weights associated to the case  of matrices having the important property as follows: there exists ${\rho _{\vec{A} }} \geq 1$ such that 
\begin{equation}\label{DK1}
\left\| {{A_i}(t)} \right\|.\left\| {A_i^{ - 1}(t)} \right\| \leq {\rho _{\vec{A} }}, \,\,\,{\text{\rm for all  }}\,i=1,...,m,
\end{equation}
for almost everywhere $t \in \mathbb R^n$. Thus, by the property of invertible matrice, it is easy  to show that
\begin{equation}\label{DK1.1}
\left\| A_i(t) \right\|^\sigma \lesssim \left\| A_i^{-1}(t)\right\|^{-\sigma},\,\textit{\rm for all }\,\sigma \,\in \mathbb R,
\end{equation}
and
\begin{equation}\label{DK1.2}
|A_i(t)x|^\sigma \gtrsim \left\| A_i^{-1}(t)\right\|^{-\sigma}.|x|^{\sigma}, \,\textit{\rm for all }\, \sigma\in\mathbb R,x\in \mathbb R^n.
\end{equation}
{\bfseries Remark 2.} If  $A(t)=(a_{ij}(t))_{n\times n}$ is a real orthogonal matrix for almost everywhere $t$ in $\mathbb R^n$, then $A(t)$ satisfies the property (\ref{DK1}). Indeed, we know that the definition of the matrix norm (\ref{normB}) is the special case of Frobenius matrix norm. We recall the property of the above norm as follows
\[
\sqrt{\rho(B^*.B)}\leq \big\|B\big\|\leq \sqrt{n}.\sqrt{\rho(B^*.B)},\,\, {\text {for all}}\, \,B\in \, M_n(\mathbb C),
\]
where $B^*$ is the conjugate matrix of $B$ and $\rho(B)$ is the largest modulus of the eigenvalues of $B$. Thus, since $A^{-1}(t)$ is also a real orthogonal matrix, we get
$$ \left\|A(t)\right\|\leq \sqrt{n}\,\,\text{\rm and}\,\, \left\|A^{-1}(t)\right\|\leq \sqrt{n},$$
which immediately obtain the desired result. Now, we are ready to state our first main result in this paper.

\begin{theorem}\label{TheoremVarLebesgue}
Let $\omega_1(x)=|x|^{\gamma_1}, ... , \omega_m(x)=|x|^{\gamma_m}, \omega(x)=|x|^{\gamma}$, $q(\cdot)\in \mathcal P_b(\mathbb R^n)$, $\zeta >0$ and the following conditions are true:
\begin{equation}\label{DKVarLebesgue}
q_i(A_i^{-1}(t)\cdot)\leq \zeta.q_i(\cdot)\,\textit{\rm and}\,\, \big\|1\big\|_{L^{r_{i}(t,\cdot)}}<\infty, \, \text{a.e. }\,t\in \textit{\rm supp}(\Phi),
\end{equation}
\begin{equation}\label{DKVarLebesgue1}
\mathcal C_1= \int\limits_{\mathbb R^n}\frac{\Phi(t)}{|t|^n}\prod\limits_{i=1}^{m} c_{A_i,q_i,\gamma_i}(t).\big\|1\big\|_{L^{r_{i}(t,\cdot)}}dt<\infty,
\end{equation}
where 
\[
{c_{A_i,q_i,\gamma_i}(t)} = {{{\rm max }}\big\{ {{{\big\| {{A_i}(t)} \big\|}^{ - {\gamma _i}}},{{\big\| {A_i^{ - 1}(t)} \big\|}^{{\gamma _i}}}} \big\}}{\rm max}\big\{ {\left| {\det A_i^{ - 1}(t)} \right|^{\frac{1}{q_{i+}}}},{\left| {\det A_i^{ - 1}(t)} \right|^{\frac{1}{q_{i-}}}}\big\},
\]
\[
\dfrac{1}{r_i(t,\cdot)}=\dfrac{1}{q_i(A_i^{-1}(t)\cdot)}-\dfrac{1}{\zeta q_i(\cdot)},\,\textit{\rm for all} \,i=1,...,m.
\]
Then, $H_{\Phi,\vec{A}}$ is a bounded operator from $L^{\zeta q_1(\cdot)}_{\omega_1}\times\cdots\times L^{\zeta q_m(\cdot)}_{\omega_m}$ to $L^{ q(\cdot)}_{\omega}$.
\end{theorem}
\begin{proof}
By using the versions of the Minkowski inequality for variable Lebesgue spaces from Corollary 2.38 in \cite{CUF2013}, we have
\begin{equation}\label{0HALq}
\big\|H_{\Phi,\vec{A}}(\vec{f})\big\|_{L^{q(\cdot)}_{\omega}}\lesssim\int\limits_{\mathbb R^n}{\frac{\Phi(t)}{|t|^n}\big\|\prod\limits_{i=1}^{m}f(A_i(t).)\big\|_{L^{q(\cdot)}_{\omega}}}dt.
\end{equation}
By assuming $\sum\limits_{i=1}^m{\frac{1}{q_i(\cdot)}}=\frac{1}{q(\cdot)}$ and applying the H\"{o}lder inequality for variable Lebesgue spaces (see also Corollary 2.28  in \cite{CUF2013}), we imply that 
\[
\big\|\prod\limits_{i=1}^{m}f_i(A_i(t).)\big\|_{L^{q(\cdot)}_{\omega}}\lesssim \prod_{i=1}^{m}\big\|f_i(A_i(t).)\big\|_{L^{q_i(\cdot)}_{\omega_i}}.
\]
Consequently, we obtain
\begin{equation}\label{HALqLeb}
\big\|H_{\Phi,\vec{A}}(\vec{f})\big\|_{L^{q(\cdot)}_{\omega}}\lesssim \int\limits_{\mathbb R^n}{\frac{\Phi(t)}{|t|^n}\prod_{i=1}^{m}\big\|f_i(A_i(t).)\big\|_{L^{q_i(\cdot)}_{\omega_i}}}dt.
\end{equation}
For $\eta>0$, we see that
\begin{eqnarray}
 &&\int\limits_{\mathbb R^n} {\left(\dfrac{\big|{{f_i}({A_i}(t).x)}\big|{\omega_i}(x)}{\eta}\right)^{q_i(x)}dx}
\nonumber
\\
&&= \int\limits_{\mathbb R^n} {{\left(\dfrac{\big| {{f_i}(z)}\big|{{\big| {A_i^{ - 1}(t)z} \big|}^{{\gamma_i}}}}{\eta}\right)^{q_i(A_i^{-1}(t).z)}}\big| {\det A_i^{-1}(t)}\big|dz}
 \nonumber
 \\
 &&\leq \big| {\det A_i^{ - 1}(t)} \big|.\int\limits_{\mathbb R^n} {{\left(\dfrac{\max \big\{ {{{\left\| {A_i^{ - 1}(t)} \right\|}^{{\gamma_i}}},{{\left\| {{A_i}(t)} \right\|}^{ - {\gamma_i}}}} \big\}\big| {{f_i}(z)} \big|{\omega_i(z)}}{\eta}\right)^{{q_i(A_i^{-1}(t).z)}}}dz}\nonumber
 \\
 &&\leq\int\limits_{\mathbb R^n} {{\left(\dfrac{c_{A_i,q_i,\gamma_i}(t).\big| {{f_i}(z)} \big|{\omega_i(z)}}{\eta}\right)^{{q_i(A_i^{-1}(t).z)}}}dz}.\nonumber
\end{eqnarray}
Hence, it follows from the definition of the Lebesgue space with variable exponent that
\begin{equation}\label{MVarfiBLeb}
\big\|f_i(A_i(t).)\big\|_{L^{q_i(\cdot)}_{\omega_i}}\leq c_{A_i,q_i,\gamma_i}(t).\big\|f_i\big\|_{L^{q_i(A_i^{-1}(t)\cdot)}_{\omega_i}}.
\end{equation}
In view of ($\ref{DKVarLebesgue}$) and Theorem \ref{theoembed}, we deduce
\begin{equation}\label{nhung2}
\big\|f\big\|_{L^{q_i(A_i^{-1}(t)\cdot)}_{\omega_i}}\lesssim \big\|1\big\|_{L^{r_i(t,\cdot)}}.\big\|f\big\|_{L^{\zeta q_i(\cdot)}_{\omega_i}},
\end{equation} 
Therefore, by (\ref{HALqLeb})-(\ref{nhung2}), we obtain
\[
\big\|H_{\Phi,\vec{A}}(\vec{f})\big\|_{L^{q(\cdot)}_{\omega}}\lesssim \mathcal C_1.\prod\limits_{i=1}^m\big\|f_i\big\|_{L^{\zeta q_i(\cdot)}_{\omega_i}},
\]
which finishes the proof of this theorem.
\end{proof}
In particular, when $\zeta \leq 1$, we have the following important result to the case  of matrices having property (\ref{DK1}) above.
\begin{theorem}\label{TheoremVarLebesgue1}
Let us have the given supposition of Theorem \ref{TheoremVarLebesgue} and assume that
\begin{equation}\label{DKVarLeb}
q_i(A_i^{-1}(t)\cdot)\leq q_i(\cdot), \big\|1\big\|_{L^{r_{1i}(t,\cdot)}}<\infty, 
\end{equation}
where $\dfrac{1}{r_{1i}(t,\cdot)}=\dfrac{1}{q_i(A_i^{-1}(t)\cdot)}-\dfrac{1}{q_i(\cdot)}, \, \text{a.e. }\,t\in \textit{\rm supp}(\Phi)$, for all $i=1,...,m$. 
\\
{\rm(a)} If 
\[
\mathcal C_2 =\int\limits_{\mathbb R^n}{\dfrac{\Phi(t)}{|t|^n}\prod\limits_{i=1}^{m}\max\Big\{\big\|A_i^{-1}(t)\big\|^{\frac{n}{q_{i+}}+\gamma_i}, \big\|A_i^{-1}(t)\big\|^{\frac{n}{q_{i-}}+\gamma_i}\Big\}\big\|1\big\|_{L^{r_{1i}(t,\cdot)}}dt}<\infty,
\]
then 
\[
\big\|H_{\Phi,\vec A}(\vec f)\big\|_{L^{ q(\cdot)}_{\omega}}\lesssim \mathcal C_2\prod\limits_{i=1}^{m}\big\| f_i\big\|_{L^{q_i(\cdot)}_{\omega_i}}.
\]
{\rm(b)} Let
\[
\mathcal C_2^* =\int\limits_{\mathbb R^n}{\dfrac{\Phi(t)}{|t|^n}\prod\limits_{i=1}^{m}\min\Big\{\big\|A_i^{-1}(t)\big\|^{\frac{n}{q_{i+}}+\gamma_i}, \big\|A_i^{-1}(t)\big\|^{\frac{n}{q_{i-}}+\gamma_i}\Big\}}dt.
\]
Assume that $H_{\Phi,\vec{A}}$ is a bounded operator from $L^{q_1(\cdot)}_{\omega_1}\times\cdots\times L^{q_m(\cdot)}_{\omega_m}$ to $L^{ q(\cdot)}_{\omega}$ and the following condition is satisfied:
\begin{equation}\label{dangthucVarLeb}
\dfrac{1}{q_{1-}}+ \dfrac{1}{q_{2-}}+\cdots+\dfrac{1}{q_{m-}}=\dfrac{1}{q_+}.
\end{equation}
Then, $\mathcal C_2^*$ is finite. Moreover,
\[
\big\|H_{\Phi,\vec A}\big\|_{L^{q_1(\cdot)}_{\omega_1}\times\cdots\times L^{q_m(\cdot)}_{\omega_m} \to L^{ q(\cdot)}_{\omega}}\gtrsim \mathcal C_2^*.
\]
\end{theorem}
\begin{proof}
We begin with the proof for the case ${\rm(a)}$. From (\ref{DKVarLeb}), by using the  Theorem \ref{theoembed} again, we have
\begin{equation}\label{nhungVarLeb2}
\big\|f\big\|_{L^{q_i(A_i^{-1}(t).)}_{\omega_i}}\lesssim \big\|1\big\|_{L^{r_{1i}(t,\cdot)}}.\big\|f\big\|_{L^{q_i(.)}_{\omega_i}},\,\textit{\rm for all}\, f \in L^{q_i(A_i^{-1}(t).)}_{\omega_i}.
\end{equation}
On the other hand, by (\ref{detA}) and (\ref{DK1.1}), for $i = 1, 2, ..., m$, we find
\begin{equation}\label{cAiq+-}
c_{A_i,q_i,\gamma_i}(t)\lesssim \max\Big\{\big\|A_i^{-1}(t)\big\|^{\frac{n}{q_{i+}}+\gamma_i}, \big\|A_i^{-1}(t)\big\|^{\frac{n}{q_{i-}}+\gamma_i}\Big\}.
\end{equation}
By the similar arguments as Theorem \ref{TheoremVarLebesgue}, by (\ref{nhungVarLeb2}) and (\ref{cAiq+-}), we also obtain
\[
\big\|H_{\Phi,\vec A}(\vec f)\big\|_{L^{q(\cdot)}_{\omega}}\lesssim \mathcal C_2\prod\limits_{i=1}^m\big\|f_i\big\|_{L^{q_i(\cdot)}_{\omega_i}}.
\]
Now, for the case ${\rm(b)}$, we make the functions $f_i$ for all $i=1,...m$ as follows:
\[{f_i}(x) =
\begin{cases}
0,\,\,\;\;\;\;\;\,\,\,\,\,\,\,\,\,\,\,\,\,\,\,\,\,\,\,\,{\rm if }\, | x | < \rho_{\vec A}^{ - 1},&\\
  {| x |^{ - \frac{{n}}{{{q_i(x)}}}-\gamma_i - \varepsilon }},\,{\rm otherwise.}&
\end{cases}
\]
This immediately deduces that $\big\|f_i\big\|_{L^{q_i(\cdot)}_{\omega_i}} >0$. Besides that, we also compute
\begin{eqnarray}
F_{q_i}(f_i\omega_i)&=&\int\limits_{|x|\geq\rho_{\vec A}^{-1}}{|x|^{-n-\varepsilon q_i(x)}}dx=\int\limits_{\rho_{\vec A}^{-1}}^{+\infty}\int\limits_{S^{n-1}}{r^{-\varepsilon q_i(rx')-1}}
d\sigma(x')dr\nonumber
\\
&=& \int\limits_{\rho_{\vec A}^{-1}}^{1}\int\limits_{S^{n-1}}{r^{-\varepsilon q_i(rx')-1}}d\sigma(x')dr
+ \int\limits_{1}^{+\infty}\int\limits_{S^{n-1}}{r^{-\varepsilon q_i(rx')-1}}d\sigma(x')dr.
\nonumber
\end{eqnarray}
Thus, $F_{q_i}(f_i\omega_i)$ is dominated by
\[
\int\limits_{\rho_{\vec A}^{-1}}^{1}\int\limits_{S^{n-1}}{r^{-1-\varepsilon q_{i+}}}d\sigma(x')dr
+ \int\limits_{1}^{+\infty}\int\limits_{S^{n-1}}{r^{-1-\varepsilon q_{i-}}}d\sigma(x')dr
\lesssim \Big((\rho_{\vec A}^{\varepsilon q_{i+}}-1)q_{i-}+q_{i+}\Big){\varepsilon}^{-1}.
\]
From the above estimation, by (\ref{maxminvar}), we get 
\begin{equation}\label{normfiVarLeb}
\big\|f_i\big\|_{L^{q_i(\cdot)}_{\omega_i}} \lesssim \Big((\rho_{\vec A}^{\varepsilon q_{i+}}-1)q_{i-}+q_{i+}\Big)^{\frac{1}{q_{i-}}}{\varepsilon}^{\frac{-1}{q_{i-}}}.
\end{equation}
Next, let us denote two useful sets as follows
\[
{S_x} = \bigcap\limits_{i = 1}^m {\big\{ {t \in {\mathbb R^n}: | {{A_i}(t)}x| \geq {\rho^{-1}_{\vec A}} \big\}}},
\]
and 
\[
U = \big\{ {t \in {\mathbb R^n}:\big\|A_i(t)\big\| \geq \varepsilon,\,\text{for all }\, i = 1,...,m} \big\}.\]
Then, we claim that
\begin{equation}\label{USx'Leb}
 U \subset S_x,\,\,\text{for all}\,\, x\in\mathbb R^n\setminus B(0,\varepsilon^{-1}).
\end{equation}
Indeed, let $t\in U$. This leads that $\big\|A_i(t)\big\|.|x|\geq 1,\,\text{for all}\,\, x\in\mathbb R^n\setminus B(0,\varepsilon^{-1}).$ Therefore, it follows from applying the condition (\ref{DK1}) that
 \[
|A_i(t)x|\geq \big\|A_i^{-1}(t)\big\|^{-1}|x|\geq\rho_{\vec A}^{-1},
\]
which finishes the proof of the relation (\ref{USx'Leb}). 
\\
Now, by letting $x\in \mathbb R^n\setminus B(0,\varepsilon^{-1})$ and using (\ref{USx'Leb}), we see that
\[
{H_{\Phi ,\vec A }}(\vec{f)} (x) \geq \int\limits_{S_{x}} {\frac{{\Phi (t)}}{{{{\left| t \right|}^n}}}} \prod\limits_{i = 1}^m {{{| {{A_i}(t)x}|}^{- \frac{{n}}{{{q_i(x)}}} - \gamma_i-\varepsilon }}} dt \geq \int\limits_{U} {\frac{{\Phi (t)}}{{{{\left| t \right|}^n}}}} \prod\limits_{i = 1}^m {{{| {{A_i}(t)x}|}^{- \frac{{n}}{{{q_i(x)}}} - \gamma_i-\varepsilon }}} dt.
\]
Thus, by (\ref{DK1.2}), we find
\begin{equation}\label{HAf3Leb}
{H_{\Phi ,\vec A }}(\vec{f)} (x) \gtrsim \Big(\int\limits_{U} {\frac{{\Phi (t)}}{{{{\left| t \right|}^n}}}} \prod\limits_{i = 1}^m {{{\big\| A_i^{-1}(t)\big\|}^{\frac{{n}}{{{q_i(x)}}}+\gamma_i+\varepsilon }}} dt\Big)|x|^{- \frac{{n}}{{{q(x)}}} - \gamma-m\varepsilon}\chi_{\mathbb R^n\setminus B(0,\varepsilon^{-1})}(x).
\end{equation}
For convenience, we denote
\[
\Gamma_\varepsilon=\int\limits_{U}{\dfrac{\Phi(t)}{|t|^n}\prod\limits_{i=1}^{m}\min\Big\{\big\|A_i^{-1}(t)\big\|^{\frac{n}{q_{i+}}+\gamma_i}, \big\|A_i^{-1}(t)\big\|^{\frac{n}{q_{i-}}+\gamma_i}\Big\}}\prod\limits_{i=1}^m \big\|A_i^{-1}(t)\big\|^{\varepsilon}\varepsilon^{m\varepsilon}dt.
\]
Hence, by (\ref{HAf3Leb}), we arrive at
\begin{eqnarray}\label{HAf4Leb}
\big\|{H_{\Phi ,\vec A }}(\vec{f)} \big\|_{L^{q(\cdot)}_{\omega}}
&&\gtrsim \varepsilon^{-m\varepsilon}\Gamma_{\varepsilon}.\big\| |\cdot|^{- \frac{{n}}{{{q(\cdot)}}}-\gamma - m\varepsilon}\chi_{\mathbb R^n\setminus B(0,\varepsilon^{-1})}\big\|_{L^{q(\cdot)}_{\omega}}\nonumber\\
&&=:\varepsilon^{-m\varepsilon}\Gamma_{\varepsilon}.\big\| h\big\|_{L^{q(\cdot)}_{\omega}},
\end{eqnarray}
where $h(x)=|x|^{- \frac{{n}}{{{q(x)}}}-\gamma - m\varepsilon}\chi_{\mathbb R^n\setminus B(0,\varepsilon^{-1})}$.
Next, we will prove the following result
\begin{equation}\label{chuanfLeb}
\big\| h \big\|_{L^{q(\cdot)}_{\omega}}\gtrsim \varepsilon^{m\varepsilon\frac{q_{+}}{q_{-}}}.\varepsilon^{\frac{-1}{q_+}}.
\end{equation}
Indeed, for $\varepsilon$ sufficiently small such that $\varepsilon^{-1}>1$, we compute
\begin{eqnarray}
F_{q}(h\omega)&=&\int\limits_{|x|\geq \varepsilon^{-1}}{|x|^{-n-m\varepsilon q(x)}}dx= \int\limits_{\varepsilon^{-1}}^{+\infty}\int\limits_{S^{n-1}}{r^{-1-m\varepsilon q(rx')}}
d\sigma(x')dr\nonumber
\\
&\geq& \int\limits_{\varepsilon^{-1}}^{+\infty}\int\limits_{S^{n-1}}{r^{-1-m\varepsilon q_{+}}}
d\sigma(x')dr\gtrsim \varepsilon^{m\varepsilon q_{+}}.\varepsilon^{-1}.
\end{eqnarray}
From this, by the inequality (\ref{maxminvar}), we immediately obtain the inequality (\ref{chuanfLeb}).
By writing $\vartheta(\varepsilon)$ as
\[
\vartheta \big(\varepsilon\big) = \dfrac{\varepsilon^{m\varepsilon\frac{q_{+}}{q_{-}}}.\varepsilon^{\frac{-1}{q_+}}}{\prod\limits_{i=1}^{m}\Big((\rho_{\vec A}^{\varepsilon q_{i+}}-1)q_{i-}+q_{i+}\Big)^{\frac{1}{q_{i-}}}{\varepsilon}^{\frac{-1}{q_{i-}}}},
\]
then, by (\ref{HAf4Leb}) and (\ref{chuanfLeb}), we estimate
\begin{equation}\label{HAf5Leb}
\big\| {{H_{\Phi ,\vec A }}\big( {\vec f } \big)} \big\|_{L^{q(\cdot)}_{\omega}}\gtrsim \varepsilon^{-m\varepsilon}\vartheta.\Gamma_\varepsilon.\prod\limits_{i=1}^m\big\|f_i\big\|_{L^{q_i(\cdot)}_{\omega_i}}.
\end{equation}
Note that, by  lettting $\varepsilon$  sufficiently small and $t\in U$, we find
\begin{equation}\label{prodLeb}
 \prod\limits_{i=1}^m \big\|A_i^{-1}(t)\big\|^\varepsilon.\varepsilon^{m\varepsilon}\leq \rho_{\vec A}^{\varepsilon}\lesssim 1.
\end{equation}
By the relation (\ref{dangthucVarLeb}), we get the limit of function $\varepsilon^{-m\varepsilon}\vartheta$ is a positive number when $\varepsilon$ tends to zero. Thus, by (\ref{HAf5Leb}), (\ref{prodLeb}) and the dominated convergence theorem of Lebesgue, we also have 
\[
\int\limits_{\mathbb R^n}{\dfrac{\Phi(t)}{|t|^n}\prod\limits_{i=1}^{m}\min\Big\{\big\|A_i^{-1}(t)\big\|^{\frac{n}{q_{i+}}+\gamma_i}, \big\|A_i^{-1}(t)\big\|^{\frac{n}{q_{i-}}+\gamma_i}\Big\}}dt<\infty.
\]
which completes the proof of the theorem.
\end{proof}
Next, we discuss the boundedness of the multilinear Hausdorff operators on the product of weighted Morrey-Herz spaces with variable exponent.
\begin{theorem}\label{TheoremVar-MorreyHerz}
Let $\omega_1(x)=|x|^{\gamma_1}$,..., $\omega_m(x)=|x|^{\gamma_m}$, $\omega(x)= |x|^{\gamma}$, $q(\cdot)\in\mathcal P_b(\mathbb R^n)$, $\lambda_1,...,\lambda_m,\zeta >0$, and the hypothesis (\ref{DKVarLebesgue}) in Theorem \ref{TheoremVarLebesgue} hold. Suppose that  for all $i =1,...,m $, we have
\begin{equation}\label{Var-MorreyHerzDK3}
\alpha_i(0)-\alpha_{i\infty} \geq 0.
\end{equation}
At the same time, let
\begin{eqnarray}\label{Var-MorreyHerzDK4}
\mathcal C_3&=& \int\limits_{\mathbb R^n}{\frac{\Phi(t)}{|t|^n}\prod\limits_{i=1}^{m}c_{A_i,q_i,\gamma_i}(t)\big\|1\big\|_{L^{r_{i}(t,\cdot)}}.{\max\Big\{\big\|A_i(t)\big\|^{\lambda_i-\alpha_i(0)}, \big\|A_i(t)\big\|^{\lambda_i-\alpha_{i\infty}} \Big\}}}\times
\nonumber
\\
&&\,\,\,\,\,\,\,\times{\max\Big\{\sum\limits_{r=\Theta_n^*-1}^{0} 2^{r(\lambda_i-\alpha_i(0))},\sum\limits_{r=\Theta_n^*-1}^{0}2^{r(\lambda_i-\alpha_{i\infty})}\Big\}}dt<\infty,
\nonumber
\\
\end{eqnarray}
where $\Theta_n^*=\Theta_n^*(t)$ is the greatest integer number satisfying 
\[
\mathop{\rm max}\limits_{i=1,...,m}\big\{\|A_i(t)\|.\|A_i^{-1}(t)\|\big\}<2^{-\Theta_n^*},\,\, \text{\rm for a.e. }\, t\in\mathbb R^n.
\]
Then, $H_{\Phi,\vec{A}}$ is a bounded operator from ${M{\mathop K\limits^.}}^{\alpha_1(\cdot),\lambda_1}_{p_1,\zeta q_1(\cdot),\omega_1}\,\times\cdots\times {M{\mathop K\limits^.}}^{\alpha_m(\cdot),\lambda_m}_{p_m,\zeta q_m(\cdot),\omega_m}$ to ${M{\mathop K\limits^.}}^{\alpha(\cdot),\lambda}_{p, q(\cdot),\omega}$.
\end{theorem}
\begin{proof}
By estimating as (\ref{HALqLeb}) above, we have
\begin{equation}\label{HALq}
\big\|H_{\Phi,\vec{A}}(\vec{f})\chi_k\big\|_{L^{q(\cdot)}_{\omega}}\lesssim \int\limits_{\mathbb R^n}{\frac{\Phi(t)}{|t|^n}\prod_{i=1}^{m}\big\|f_i(A_i(t).)\chi_k\big\|_{L^{q_i(\cdot)}_{\omega_i}}}dt.
\end{equation}
Let us now fix $i\in \big\{1,2,...,m\big\}$. Since  $\|A_i(t)\|\neq 0$, there exists an integer number $\ell_i=\ell_i(t)$ such that $2^{\ell_i-1}<\|A_i(t)\|\leq 2^{\ell_i}$. For simplicity of notation, we  write $\rho_{\vec A}^* (t) =\mathop{\rm max}\limits_{i=1,...,m}\big\{\big\|A_i(t)\big\|.\big\|A_i^{-1}(t)\big\|\big\}$. Then, by letting $y=A_i(t).z$ with $z\in C_k$, it follows that 
\[
| y | \geq {\left\| {A_i^{ - 1}(t)} \right\|^{ - 1}}\left| z \right|\geq \frac{{{2^{{\ell_i} + k - 2}}}}{{{\rho _{\vec{A}}^*}}} > {2^{k + {\ell_i} - 2 + {\Theta_n^*}}},
\]
and
\[
| y |\leq \left\| {{A_i}(t)} \right\|.\left| z \right| \leq {2^{{\ell_i} + k}}.
\]
These estimations can be used to get
\begin{equation}\label{AiCk}
{A_i}(t).{C_k} \subset \left\{ {z \in {\mathbb R^n}:{2^{k + {\ell_i} - 2 + {\Theta _n^*}}} < \left| z \right| \leq {2^{k + {\ell_i}}}} \right\}.
\end{equation}
Now, we need to prove that
\begin{equation}\label{fAchikLq}
\big\|f_i(A_i(t).)\chi_k\big\|_{L^{q_i(\cdot)}_{\omega_i}}\lesssim c_{A_i,q_i,\gamma_i}(t)\big\|1 \big\|_{L^{r_i(t,\cdot)}}.\sum\limits_{r=\Theta_n^*-1}^{0} \big\|f_i\chi_{k+\ell_i+r}\big\|_{L^{\zeta q_i(\cdot)}_{\omega_i}}.
\end{equation}
Indeed, for $\eta>0$, by (\ref{AiCk}), we find
\begin{eqnarray}
&&\int\limits_{\mathbb R^n}{\left(\dfrac{\big|f_i(A_i(t)x)\big|\chi_k(x)\omega_i(x)}{\eta}\right)^{q_i(x)}dx}\nonumber
\\
&&=\int\limits_{A_i(t)C_k}{\left(\dfrac{\big|f_i(z)\big|\big|A_i^{-1}(t)z\big|^{\gamma_i}}{\eta}\right)^{q_i(A_i^{-1}(t)z)}\big|
\textit{\rm det} A_i^{-1}(t)\big|dz}\nonumber
\\
&&\leq \big|\textit{\rm det} A_i^{-1}(t)\big|\int\limits_{A_i(t)C_k}\left(\frac{{\rm max}\big\{\big\|A_i^{-1}(t)\big\|^{\gamma_i},\big\|A_i(t)\big\|^{-\gamma_i}\big\}\big| f_i(z)\big|\omega_i(z)}{\eta}\right)^{q_i(A_i^{-1}(t).z)}dz.
\nonumber
\end{eqnarray}
So, we have that
\begin{eqnarray}\label{fAeta}
&&\int\limits_{\mathbb R^n}{\left(\dfrac{\big|f_i(A_i(t)x)\chi_k(x)\big|\omega_i(x)}{\eta}\right)^{q_i(x)}dx}
\nonumber
\\
&&\leq\int\limits_{\mathbb R^n}\left(\frac{c_{A_i,q_i,\gamma_i}(t)\big|\sum\limits_{r=\Theta_n^*-1}^0 f_i(z)\chi_{k+\ell_i+r}(z)\big|\omega_i(z)}{\eta}\right)^{q_i(A_i^{-1}(t).z)}dz.
\nonumber
\end{eqnarray}
Therefore, by the definition of Lebesgue space with variable exponent, it is easy  to get that
\[
\big\|f_i(A_i(t).)\chi_k\big\|_{L^{q_i(\cdot)}_{\omega_i}}\leq c_{A_i,q_i,\gamma_i}(t).\sum\limits_{r=\Theta_n^*-1}^{0} \big\|f_i\chi_{k+\ell_i+r}\big\|_{L^{q_i(A_i^{-1}(t)\cdot)}_{\omega_i}},
\]
which completes the proof of the inequalities (\ref{fAchikLq}), by (\ref{nhung2}). Now, it immediately follows from (\ref{HALq}) and (\ref{fAchikLq})  that
\begin{eqnarray}\label{LemmaVeH}
\big\|H_{\Phi,\vec A}(\vec f)\chi_k\big\|_{L^{q(\cdot)}_{\omega}}&\lesssim & \int\limits_{\mathbb R^n}{\frac{\Phi(t)}{|t|^n}\prod\limits_{i=1}^{m}c_{A_i,q_i,\gamma_i}(t)\big\|1\big\|_{L^{r_i(t,\cdot)}}}\times
\nonumber
\\
&&\,\,\,\,\,\,\,\,\times\prod\limits_{i=1}^{m}\Big(\sum\limits_{r=\Theta_n^*-1}^{0} \big\|f\chi_{k+\ell_i+r}\big\|_{L^{\zeta q_i(\cdot)}_{\omega_i}}\Big)dt.
\end{eqnarray}
Consequently, by  applying Lemma \ref{lemmaVE} in Section 2, we get
\begin{eqnarray}\label{HlemmaVe}
\big\|H_{\Phi,\vec A}(\vec f)\chi_k\big\|_{L^{q(\cdot)}_{\omega}}&\lesssim& \int\limits_{\mathbb R^n}{\frac{\Phi(t)}{|t|^n}\mathcal L(t).\prod\limits_{i=1}^{m}c_{A_i,q_i,\gamma_i}(t)\big\|1\big\|_{L^{r_i(t,\cdot)}}}\times\nonumber
\\
&&\,\,\,\,\,\,\,\,\,\,\,\,\,\,\,\,\,\,\,\,\,\,\,\,\,\,\,\,\,\,\times\prod_{i=1}^{m}\big\|f_i\big\|_{M{\mathop K\limits^.}^{\alpha_i(\cdot), \lambda_i}_{p_i,\zeta q_i(\cdot),\omega_i}}dt,
\end{eqnarray}
where
\[
\mathcal L(t) =\prod\limits_{i=1}^{m}\Big(2^{(k+\ell_i)(\lambda_i-\alpha_i(0))}\sum\limits_{r=\Theta_n^*-1}^{0} 2^{r(\lambda_i-\alpha_i(0))}+ 2^{(k+\ell_i)(\lambda_i-\alpha_{i\infty})}\sum\limits_{r=\Theta_n^*-1}^{0}2^{r(\lambda_i-\alpha_{i\infty})}\Big).
\]
By having $2^{\ell_i-1}<\big\|A_i(t)\big\|\leq 2^{\ell_i}$, for all $i=1,...,m$, it implies that
\[
2^{\ell_i(\lambda_i-\alpha_i(0))}+2^{\ell_i(\lambda_i-\alpha_{i\infty})}\lesssim \max\big\{\big\|A_i(t)\big\|^{\lambda_i-\alpha_i(0)}, \big\|A_i(t)\big\|^{\lambda_i-\alpha_{i\infty}} \big\}.
\]
Thus, we can estimate $\mathcal L$ as follows
\begin{eqnarray}
\mathcal L(t)&\lesssim &\prod\limits_{i=1}^{m} {\max\Big\{\big\|A_i(t)\big\|^{\lambda_i-\alpha_i(0)}, \big\|A_i(t)\big\|^{\lambda_i-\alpha_{i\infty}} \Big\}}\times\nonumber
\\
&&\times\Big\{2^{k(\lambda_i-\alpha_i(0))}\sum\limits_{r=\Theta_n^*-1}^{0} 2^{r(\lambda_i-\alpha_i(0))}+ 2^{k(\lambda_i-\alpha_{i\infty})}\sum\limits_{r=\Theta_n^*-1}^{0}2^{r(\lambda_i-\alpha_{i\infty})}\Big\}\nonumber
\\
&\lesssim & \prod\limits_{i=1}^{m} {\max\Big\{\big\|A_i(t)\big\|^{\lambda_i-\alpha_i(0)}, \big\|A_i(t)\big\|^{\lambda_i-\alpha_{i\infty}} \Big\}}\times\nonumber
\\
&&\times{\max\Big\{\sum\limits_{r=\Theta_n^*-1}^{0} 2^{r(\lambda_i-\alpha_i(0))},\sum\limits_{r=\Theta_n^*-1}^{0} 2^{r(\lambda_i-\alpha_{i\infty})}\Big\}}\Big\{2^{k(\lambda_i-\alpha_i(0))}+ 2^{k(\lambda_i-\alpha_{i\infty})}\Big\}.
\nonumber
\end{eqnarray}
From this, by (\ref{HlemmaVe}), it is not difficult to show that
\begin{equation}\label{HAchikLq}
\big\|H_{\Phi,\vec A}(\vec f)\chi_k\big\|_{L^{q(\cdot)}_{\omega}}\lesssim \mathcal C_3\prod_{i=1}^{m}\big(2^{k(\lambda_i -\alpha_i(0))}+2^{k(\lambda_i -\alpha_{i\infty})}\big).\prod_{i=1}^{m}\big\|f_i\big\|_{M{\mathop K\limits^.}^{\alpha_i(\cdot), \lambda_i}_{p_i,\zeta q_i(\cdot),\omega_i}}.
\end{equation}
On the other hand, using Proposition 2.5 in \cite{LZ2014}, we get
\begin{equation}\label{HAfMK}
\big\| H_{\Phi,\vec A}(\vec f)\big\|_{M{\mathop K\limits^.}^{\alpha(\cdot), \lambda}_{p,q(\cdot),\omega}} \lesssim \max\big\{
\sup\limits_{k_0<0,k_0\in\mathbb Z}E_1, \sup\limits_{k_0\geq 0,k_0\in\mathbb Z}(E_2+E_3)
\big\},
\end{equation}
where 
\begin{eqnarray}
&&E_1=2^{-k_0\lambda}\Big(\sum\limits_{k=-\infty}^{k_0}2^{k\alpha(0)p}\big\|H_{\Phi,\vec A}(\vec f)\chi_k\big\|^p_{L^{q(\cdot)}_{\omega}}\Big)^{\frac{1}{p}},\nonumber
\\
&&E_2 = 2^{-k_0\lambda}\Big(\sum\limits_{k=-\infty}^{-1}2^{k\alpha(0)p}\big\|H_{\Phi,\vec A}(\vec f)\chi_k\big\|^p_{L^{q(\cdot)}_{\omega}}\Big)^{\frac{1}{p}},\nonumber
\\
&&E_3=2^{-k_0\lambda}\Big(\sum\limits_{k=0}^{k_0}2^{k\alpha_\infty p}\big\|H_{\Phi,\vec A}(\vec f)\chi_k\big\|^p_{L^{q(\cdot)}_{\omega}}\Big)^{\frac{1}{p}}\nonumber.
\end{eqnarray}
In order to complete the proof, it remains to estimate the upper bounds for $E_1, E_2$ and $E_3$. Note that, 
using (\ref{HAchikLq}), $E_1$ is dominated by
\[
\mathcal C_3.2^{-k_0\lambda}\Big(\sum\limits_{k=-\infty}^{k_0}2^{k\alpha(0)p}\Big(\prod_{i=1}^{m}\big(2^{k(\lambda_i -\alpha_i(0))}+2^{k(\lambda_i -\alpha_{i\infty})}\big).\prod_{i=1}^{m}\big\|f_i\big\|_{M{\mathop K\limits^.}^{\alpha_i(\cdot), \lambda_i}_{p_i,\zeta q_i(\cdot),\omega_i}}\Big)^p\Big)^{\frac{1}{p}}.
\]
This implies that
\begin{equation}\label{E1MVar}
E_1\lesssim \mathcal C_3.\mathcal T_0.\prod_{i=1}^{m}\big\|f_i\big\|_{M{\mathop K\limits^.}^{\alpha_i(\cdot), \lambda_i}_{p_i,\zeta q_i(\cdot),\omega_i}},
\end{equation}
where $\mathcal T_0= 2^{-k_0\lambda}.\Big(\sum\limits_{k=-\infty}^{k_0}2^{k\alpha(0)p}\prod\limits_{i=1}^{m}\big(2^{k(\lambda_i -\alpha_i(0))p}+2^{k(\lambda_i -\alpha_{i\infty})p}\big)\Big)^{\frac{1}{p}}$. By some simple computations, we obtain
\begin{eqnarray}
\mathcal T_0 &=& 2^{-k_0\lambda}\Big(\sum\limits_{k=-\infty}^{k_0}\prod_{i=1}^{m}\big(2^{k\lambda_i p}+2^{k(\lambda_i -\alpha_{i\infty}+\alpha_i(0))p}\big)\Big)^{\frac{1}{p}}\nonumber
\\
&\lesssim& \Big(\prod_{i=1}^{m} 2^{-k_0\lambda_i p}\Big\{  \sum\limits_{k=-\infty}^{k_0} 2^{k\lambda_i p}+ \sum\limits_{k=-\infty}^{k_0}2^{k(\lambda_i -\alpha_{i\infty}+\alpha_i(0))p}\Big\}\Big)^{\frac{1}{p}}.\nonumber
\end{eqnarray}
Hence, by assuming that $\lambda_i >0$, for all $i=1,...,m$ and (\ref{Var-MorreyHerzDK3}), we see at once that
\begin{eqnarray}
\mathcal T_0 &\lesssim& \Big(\prod_{i=1}^{m}2^{-k_0\lambda_i p}\Big\{\dfrac{2^{k_0\lambda_i p}}{1-2^{-\lambda_ip}}+ \dfrac{2^{k_0(\lambda_i -\alpha_{i\infty}+\alpha_i(0))p}}{1- 2^{-(\lambda_i -\alpha_{i\infty}+\alpha_i(0))p}}\Big\}\Big)^{\frac{1}{p}}\nonumber
\\
&\lesssim &\prod_{i=1}^{m}\Big\{\dfrac{1}{1-2^{-\lambda_ip}}+ \dfrac{2^{k_0(-\alpha_{i\infty}+\alpha_i(0))}}{1- 2^{-(\lambda_i -\alpha_{i\infty}+\alpha_i(0))p}}\Big\}\lesssim \prod\limits_{i=1}^{m}\Big(1+ 2^{{k_0}\big(\alpha_i(0)-\alpha_{i\infty}\big)}\Big).\nonumber
\end{eqnarray}
Then, from (\ref{E1MVar}), we have
\begin{equation}\label{E1}
E_1\lesssim \mathcal C_3\prod\limits_{i=1}^{m}\Big(1+ 2^{{k_0}\big(\alpha_i(0)-\alpha_{i\infty}\big)}\Big).\prod\limits_{i=1}^{m}\big\| f_i\big\|_{M{\mathop K\limits^.}^{\alpha_i(\cdot), \lambda_i}_{p_i,\zeta q_i(\cdot),\omega_i}}.
\end{equation}
By estimating in the same way as $E_1$, we also get
\begin{equation}\label{E2}
E_2\lesssim \mathcal C_3.2^{-k_0\lambda}\prod\limits_{i=1}^{m}\big\| f_i\big\|_{M{\mathop K\limits^.}^{\alpha_i(\cdot), \lambda_i}_{p_i,\zeta q_i(\cdot),\omega_i}}.
\end{equation}
For $i=1,...,m$, we denote
\[
{K_i} = \left\{ \begin{gathered}
{2^{{k_0}({\alpha _{i\infty} } - \alpha_i(0))}} + {\left| {{2^{\lambda_i p}} - 1} \right|^{ - \frac{1}{p}}} + {2^{ - {k_0}\lambda_i }},\,\textit{\rm if}\,\lambda_i  + {\alpha_{i\infty}} - \alpha_i(0) \ne 0, \hfill \\
{2^{ - {k_0}\lambda_i }}{({k_0} + 1)^{\frac{1}{p}}} + {\left| {{2^{\lambda_i p}} - 1} \right|^{ - \frac{1}{p}}},\, \rm otherwise .\hfill \\ 
\end{gathered}  \right.
\]
Then, we may show that
\begin{equation}\label{E3MVar}
E_3\lesssim \mathcal C_3\Big(\prod\limits_{i=1}^m K_i\Big).\prod\limits_{i=1}^{m}\big\| f_i\big\|_{M{\mathop K\limits^.}^{\alpha_i(\cdot), \lambda_i}_{p_i,\zeta q_i(\cdot),\omega_i}}.
\end{equation}
The proof of inequality (\ref{E3MVar}) is not difficult, but for convenience to the reader, we briefly give here. By employing (\ref{HAchikLq}) again, we make
\begin{equation}\label{E3MVar1}
E_3 \lesssim \mathcal C_3.\mathcal T_\infty .\prod_{i=1}^{m}\big\|f_i\big\|_{M{\mathop K\limits^.}^{\alpha_i(\cdot), \lambda_i}_{p_i,\zeta q_i(\cdot),\omega_i}}, 
\end{equation}
where $\mathcal T_\infty= 2^{-k_0\lambda}\Big(\sum\limits_{k=0}^{k_0}2^{k\alpha_\infty p}\prod\limits_{i=1}^{m}\big(2^{k(\lambda_i -\alpha_i(0))p}+2^{k(\lambda_i -\alpha_{i\infty})p}\big)\Big)^{\frac{1}{p}}$. By a similar argument as $\mathcal T_0$, we also get
\begin{eqnarray}
\mathcal T_\infty &\lesssim& \prod_{i=1}^{m} 2^{-k_0\lambda_i}\Big( \sum\limits_{k=0}^{k_0} 2^{k\lambda_i p}+ \sum\limits_{k=0}^{k_0}2^{k(\lambda_i +\alpha_{i\infty}-\alpha_i(0))p}\Big)^{\frac{1}{p}}\equiv \prod\limits_{i=1}^{m} \mathcal T_{i,\infty}.
\nonumber
\end{eqnarray}
In the case $\lambda_i+\alpha_{i\infty}-\alpha_i(0)\neq 0$, we deduce that $\mathcal T_{i,\infty}$ is dominated by
\begin{eqnarray}
\mathcal T_{i,\infty}&\leq& 2^{-k_0\lambda_i}\Big( \dfrac{2^{k_0\lambda_i p}-1}{2^{\lambda_i.p}-1}+\dfrac{2^{k_0(\lambda_i +\alpha_{i\infty}-\alpha_i(0))p}-1}{2^{(\lambda_i +\alpha_{i\infty}-\alpha_i(0))p}-1}\Big)^{\frac{1}{p}} \nonumber
\\
&\lesssim & {2^{{k_0}({\alpha _{i\infty} } - \alpha_i(0))}} + {\left| {{2^{\lambda_i p}} - 1} \right|^{ - 1/p}} + {2^{ - {k_0}\lambda_i }}\nonumber.
\end{eqnarray}
Otherwise, we have
\[
\mathcal T_{i,\infty} \leq \Big( \dfrac{2^{k_0\lambda_i p}-1}{2^{\lambda_i.p}-1}+(k_0+1)\Big)^{\frac{1}{p}}\lesssim  {2^{ - {k_0}\lambda_i }}{({k_0} + 1)^{\frac{1}{p}}} + {\left| {{2^{\lambda_i p}} - 1} \right|^{ - \frac{1}{p}}}.
\]
This leads that $\mathcal T_\infty\lesssim \prod\limits_{i=1}^m K_i.$ Hence, by (\ref{E3MVar1}), we  obtain  the inequality (\ref{E3MVar}).
From (\ref{HAfMK}) and (\ref{E1})-(\ref{E3MVar}), we conclude that the proof of Theorem \ref{TheoremVar-MorreyHerz} is finished.
\end{proof}

Next, we will discuss the interesting case when $\lambda_1 = \cdots = \lambda_m = 0$. Remark that these special cases of variable exponent Morrey-Herz spaces are variable exponent Herz spaces. Hence, we also have the boundedness for the multilinear Hausdorff operators on the product of weighted Herz spaces with variable exponent as follows.
\begin{theorem}\label{TheoremVarHerz}
Suppose that we have the given supposition of Theorem $\ref{TheoremVar-MorreyHerz}$ and $\alpha_i(0)=\alpha_{i\infty}$, for all $i=1,...,m$. Let $1\leq p, p_i <\infty$ such that
\begin{equation}\label{lienhieppi}
\frac{1}{p_1}+\cdots+\frac{1}{p_m}=\frac{1}{p}.
\end{equation}
 At the same time, let
\begin{eqnarray}\label{DKMorreyHerz}
\mathcal C_4 &=&\int\limits_{{\mathbb R^n}} {{(2 - {\Theta_n^*})^{m - \frac{1}{p}}}\frac{\Phi(t)}{|t|^n}\prod\limits_{i=1}^{m} c_{A_i,q_i,\gamma_i}(t)\big\|1\big\|_{L^{r_i(t,\cdot)}}}\times
\nonumber
\\
&&\,\,\,\,\,\,\,\,\,\,\,\times{{{\left\| {{A_i}(t)} \right\|}^{{ - {\alpha_i(0)}}}}}\big(\sum\limits_{r=\Theta_n^*-1}^{0}{2^{-r\alpha_i(0)}}\big)dt<\infty. 
\end{eqnarray}
Then, $H_{\Phi,\vec{A}}$ is a bounded operator from ${\mathop K\limits^.}_{\zeta q_1(\cdot),{\omega_1}}^{{\alpha _1(\cdot)}, p_1} \times \cdots\times {\mathop K\limits^.}_{\zeta q_m(\cdot),{\omega_m}}^{{\alpha _m(\cdot)}, p_m} $ to ${\mathop K\limits^.}_{q(\cdot),{\omega}}^{{\alpha(\cdot)}, p}.$
\end{theorem}
\begin{proof}
It follows from Proposition 3.8 in \cite{Almeida2012} that
\begin{eqnarray}
&&{\big\| {{H_{\Phi ,\vec A }}\big( {\vec f } \big)} \big\|_{{\mathop K\limits^.}_{q(\cdot),\omega}^{{\alpha(\cdot)}, p}}}\lesssim \Big(\sum\limits_{k=-\infty}^{-1}{2^{k\alpha(0)p}\big\|{{H_{\Phi ,\vec A }}\big( {\vec f } \big)}\chi_k\big\|_{L^{q(\cdot)}_{\omega}}^p}\Big)^{\frac{1}{p}}\nonumber
\\
&&\,\,\,\,\,\,\,\,\,\,\,\,\,\,\,\,\,\,\,\,\,\,\,\,\,\,\,\,\,\,\,\,\,\,\,\,\,\,\,\,\,\,\,\,\,\,\,\,\,\,\,\,\,\,\,\,\,\,\,\,\,\,\,\,\,+ \Big(\sum\limits_{k=0}^{\infty}{2^{k\alpha_\infty p}\big\|{{H_{\Phi ,\vec A }}\big( {\vec f } \big)}\chi_k\big\|_{L^{q(\cdot)}_{\omega}}^p}\Big)^{\frac{1}{p}}.\nonumber
\end{eqnarray}
From this, by $\alpha(0) =\alpha_\infty$, we conclude that
\begin{equation}\label{ptVarHerz}
{\big\| {{H_{\Phi ,\vec A }}\big( {\vec f } \big)} \big\|_{{\mathop K\limits^.}_{q(\cdot),\omega}^{{\alpha(\cdot)}, p}}}\lesssim \Big(\sum\limits_{k=-\infty}^{\infty}{2^{k\alpha(0)p}\big\|{{H_{\Phi ,\vec A }}\big( {\vec f } \big)}\chi_k\big\|_{L^{q(\cdot)}_{\omega}}^p}\Big)^{\frac{1}{p}}.
\end{equation}
For convenience, let us denote by 
\[
\mathcal H=\Big(\sum\limits_{k=-\infty}^{\infty}{2^{k\alpha(0)p}\big\|{{H_{\Phi ,\vec A }}\big( {\vec f } \big)}\chi_k\big\|_{L^{q(\cdot)}_{\omega}}^p}\Big)^{\frac{1}{p}}.
\] 
Next, we need to estimate  the upper bound of $\mathcal H$. By (\ref{LemmaVeH}) and using the Minkowski inequality,  we get
\begin{eqnarray}\label{H1L1}
\mathcal H &\leq& {\int\limits_{{\mathbb R^n}} {\frac{\Phi(t)}{|t|^n}\prod\limits_{i=1}^{m}c_{A_i,q_i,\gamma_i}(t)\big\|1\big\|_{L^{r_i(t,\cdot)}}}}\times
\\
&&\,\,\,\,\,\,\,\,\,\,\,\,\,\,\,\,\,\,\,\,\,\times\Big\{ {{{\sum\limits_{k =  - \infty }^{{\infty}} {{2^{k\alpha(0) p}}\prod\limits_{i = 1}^m {\Big( {\sum\limits_{r = {\Theta _n^*} - 1}^0 {\left\| {{f_i}{\chi _{k + {\ell_i} + r}}} \right\|_{L^{\zeta q_i(\cdot)}_{\omega_i}}^{}} } \Big)} } }^p}} \Big\} ^{\frac{1}{p}}dt.
\nonumber
\end{eqnarray}
By (\ref{lienhieppi}) and the H\"{o}lder inequality, it follows that
\begin{eqnarray}\label{HolderlrH1}
&&{\Big\{ {{{\sum\limits_{k = - \infty }^{{\infty}} {{2^{k\alpha(0) p}}\prod\limits_{i = 1}^m {\Big( {\sum\limits_{r = {\Theta _n^*} - 1}^0 {\left\| {{f_i}{\chi _{k + {\ell_i} + r}}} \right\|_{L^{\zeta q_i(\cdot)}_{\omega_i}}} } \Big)^p} } }}} \Big\}^{\frac{1}{p}}}
\\
&&\,\,\,\,\,\,\,\,\leq \prod\limits_{i = 1}^m {{{\Big\{ {{{\sum\limits_{k =  - \infty }^{{\infty}} {{2^{k{\alpha_i(0)}{p_i}}}\Big( {\sum\limits_{r = {\Theta_n^*} - 1}^0 {\left\| {{f_i}{\chi _{k + {\ell_i} + r}}} \right\|_{L^{\zeta q_i(\cdot)}_{\omega_i}}} } \Big)^{p_i}} }}} \Big\}}^{\frac{1}{{p_i}}}}}.\nonumber
\end{eqnarray}
On the other hand, by $p_i\geq 1$, we have
\[
{\Big( {\sum\limits_{r = {\Theta _n^*} - 1}^0 {\left\| {{f_i}{\chi _{k + {\ell_i} + r}}} \right\|_{L^{\zeta q_i(\cdot)}_{\omega_i}}^{}} } \Big)^{{p_i}}} \leq {\left( {2 - {\Theta _n^*}} \right)^{{p_i} - 1}}\sum\limits_{r = {\Theta _n^*} - 1}^0 {\left\| {{f_i}{\chi _{k + {\ell_i} + r}}} \right\|_{L^{\zeta q_i(\cdot)}_{\omega_i}}^{{p_i}}}.
\]
Hence, combining (\ref{H1L1}) and (\ref{HolderlrH1}), we obtain
\begin{equation}\label{HAf1Var}
\mathcal H \leq \int\limits_{{\mathbb R^n}} {{(2 - {\Theta_n^*})^{m - \frac{1}{p}}}\frac{\Phi(t)}{|t|^n}\prod\limits_{i=1}^{m}c_{A_i,q_i,\gamma_i}(t)\big\|1\big\|_{L^{r_i(t,\cdot)}}.\prod\limits_{i = 1}^m {{\mathcal{H}_{i}}} } dt,
\end{equation}
where $\mathcal{H}_{i} = {\sum\limits_{r = {\Theta _n^*} - 1}^0 {{{\Big( {\sum\limits_{k =  - \infty }^{{\infty}} {{2^{k{\alpha _i(0)}{p_i}}}\left\| {{f_i}{\chi _{k + {\ell_i} + r}}} \right\|_{L^{\zeta q_i(\cdot)}_{\omega_i}}^{{p_i}}} } \Big)}^{^{\frac{1}{{p_i}}}}}} }$ for all $i=1,2,...,m$. 
\\
Then, we find
\begin{eqnarray}\label{Hi}
{\mathcal{H}_{i}} &=& {\sum\limits_{r = {\Theta _n^*} - 1}^0 {{{\Big( {\sum\limits_{t =  - \infty }^{\infty} {{2^{(t-\ell_i-r){\alpha _i(0)}{p_i}}}\left\| {{f_i}{\chi _{t}}} \right\|_{L^{\zeta q_i(\cdot)}_{\omega_i}}^{{p_i}}} } \Big)}^{^{\frac{1}{{p_i}}}}}} }
\nonumber
\\
 &=&{\sum\limits_{r = {\Theta _n^*} - 1}^0 { 2^{-(\ell_i+r)\alpha_i(0)}{{\Big( {\sum\limits_{t =  - \infty }^{\infty} {{2^{t{\alpha _i(0)}{p_i}}}\left\| {{f_i}{\chi _{t}}} \right\|_{L^{\zeta q_i(\cdot)}_{\omega_i}}^{{p_i}}} } \Big)}^{^{\frac{1}{{p_i}}}}}} }
\nonumber
\\
&\leq& \big(\sum\limits_{r=\Theta_n^*-1}^{0} {2^{-r\alpha_i(0)}}\big).{2^{^{-\ell_i{\alpha_i(0)}}}}{\left\| {{f_i}} \right\|_{{\mathop K\limits^.}_{\zeta q_i(\cdot),\omega_i}^{{\alpha_i(\cdot)}, p_i}}}.
\end{eqnarray}
Since $2^{\ell_i-1}<\left\|A_i(t)\right\|\leq 2^{\ell_i}$, we imply that ${2^{^{{-\ell_i}{\alpha _i(0)}}}} \lesssim {\left\| {{A_i}(t)} \right\|^{- {\alpha_i(0)}}}$. Thus, by (\ref{HAf1Var}) and (\ref{Hi}), we get
\begin{eqnarray}
{\big\| {{H_{\Phi ,\vec A }}\big( {\vec f } \big)} \big\|_{{\mathop K\limits^.}_{q(\cdot),\omega}^{{\alpha(\cdot)}, p}}}\lesssim \mathcal C_4.\prod\limits_{i = 1}^m {{{\left\| {{f_i}} \right\|}_{{\mathop K\limits^.}_{\zeta q_i(\cdot),\omega_i}^{{\alpha_i(\cdot)}, p_i}}}},
\nonumber
\end{eqnarray}
which finishes our desired conclusion.
\end{proof}
\vskip 5pt
\hskip-15pt {\bfseries Remark 3.} We would like to give several comments on Theorem \ref{TheoremVar-MorreyHerz} and Theorem \ref{TheoremVarHerz}. If we suppose that
\[
\mathop {\rm ess\,sup}\limits_{t\in {\rm supp}(\Phi)}\big\|A_i(t)\big\|<\infty,\,\textit{\rm for all} \,i=1,...,m,
\]
then we do not need to assume the conditions $\alpha_i(0)-\alpha_{i\infty}\geq 0$ in Theorem \ref{TheoremVar-MorreyHerz} and $\alpha_i(0)=\alpha_{i\infty}$ in Theorem \ref{TheoremVarHerz}.
Indeed,  by putting 
\[
\beta= \mathop {\rm ess\,sup}\limits_{t\in {\rm supp}(\Phi)\,\textit{\rm and}\,i=1,...,m}\ell_i(t),
\] 
and applying Lemma \ref{lemmaVE} in Section 2, we refine the estimation as follow:
\\
In the case $k<\beta$, we get 
\[
\big\| f_i\chi_{k+\ell_i+r}\big\|_{L^{\zeta q_i(\cdot)}_{\omega_i}}\lesssim 2^{(k+\ell_i+r)(\lambda_i-\alpha_i(0))}\big\|f_i\big\|_{M{\mathop K\limits^.}^{\alpha_i(\cdot), \lambda_i}_{p_i,\zeta q_i(\cdot),\omega_i}}.
\]
In the case $k\geq {\beta} -\Theta_n^*+1$, we have
\[
\big\| f_i\chi_{k+\ell_i+r}\big\|_{L^{\zeta q_i(\cdot)}_{\omega_i}}\lesssim 2^{(k+\ell_i+r)(\lambda_i-\alpha_{i\infty})}\big\|f_i\big\|_{M{\mathop K\limits^.}^{\alpha_i(\cdot), \lambda_i}_{p_i,\zeta q_i(\cdot),\omega_i}}.
\]
Otherwise, then we obtain
\[
\big\| f_i\chi_{k+\ell_i+r}\big\|_{L^{\zeta q_i(\cdot)}_{\omega_i}}\lesssim \big(2^{(k+\ell_i+r)(\lambda_i-\alpha_i(0))}+2^{(k+\ell_i+r)(\lambda_i-\alpha_{i\infty})}\big)\big\|f_i\big\|_{M{\mathop K\limits^.}^{\alpha_i(\cdot), \lambda_i}_{p_i,\zeta q_i(\cdot),\omega_i}}.
\]
Also, the other estimations can be done by similar arguments as two theorems above. From this we omit details, and their proof are left to reader.
\begin{theorem}\label{TheoremVarMorreyHerz1}
Suppose that the given supposition of Theorem \ref{TheoremVar-MorreyHerz} and  the hypothesis (\ref{DKVarLeb}) in Theorem \ref{TheoremVarLebesgue1} are true.
\\
{\rm (a)} If 
\begin{eqnarray}
\mathcal C_5 &=&\int\limits_{\mathbb R^n}{\dfrac{\Phi(t)}{|t|^n}\prod\limits_{i=1}^{m}\max\Big\{\big\|A_i^{-1}(t)\big\|^{\frac{n}{q_{i+}}+\gamma_i}, \big\|A_i^{-1}(t)\big\|^{\frac{n}{q_{i-}}+\gamma_i}\Big\}\big\|A_i^{-1}(t)\big\|^{-\lambda_i}}\times\nonumber
\\
&&{{\,\,\,\,\,\,\,\,\,\,\,\,\,\,\,\,\,\,\,\,\,\,\,\,\,\,\,\,\,\,\,\,\,\,\,\,\,\,\,\times\max\Big\{\big\|A_i^{-1}(t)\big\|^{\alpha_i(0)}, \big\|A_i^{-1}(t)\big\|^{\alpha_{i\infty}} \Big\}}}\big\|1\big\|_{L^{r_{1i}(t,\cdot)}}dt<\infty,\nonumber
\end{eqnarray}
then 
\[
\big\|H_{\Phi,\vec A}(\vec f)\big\|_{M{\mathop K\limits^.}^{\alpha(\cdot), \lambda}_{p, q(\cdot),\omega}}\lesssim \mathcal C_5\prod\limits_{i=1}^{m}\big\| f_i\big\|_{M{\mathop K\limits^.}^{\alpha_i(\cdot), \lambda_i}_{p_i, q_i(\cdot),\omega_i}}.
\]
{\rm (b)} Denote by
\begin{eqnarray}
\mathcal C_5^* &=&\int\limits_{\mathbb R^n}{\dfrac{\Phi(t)}{|t|^n}\prod\limits_{i=1}^{m}\min\Big\{\big\|A_i^{-1}(t)\big\|^{\frac{n}{q_{i+}}+\gamma_i}, \big\|A_i^{-1}(t)\big\|^{\frac{n}{q_{i-}}+\gamma_i}\Big\}\big\|A_i^{-1}(t)\big\|^{-\lambda_i}}\times
\nonumber
\\
&&{{\,\,\,\,\,\,\,\,\,\,\,\,\,\,\,\,\,\,\,\,\,\,\,\,\,\,\,\,\,\,\,\,\,\,\,\,\,\,\,\,\,\,\times\min\Big\{\big\|A_i^{-1}(t)\big\|^{\alpha_i(0)+C_0^{\alpha_i}}, \big\|A_i^{-1}(t)\big\|^{\alpha_i(0)-C_0^{\alpha_i}} \Big\}}}dt.\nonumber
\end{eqnarray}
Suppose that $H_{\Phi,\vec{A}}$ is a bounded operator from ${M{\mathop K\limits^.}}^{\alpha_1(\cdot),\lambda_1}_{p_1, q_1(\cdot),\omega_1}\,\times\cdots\times {M{\mathop K\limits^.}}^{\alpha_m(\cdot),\lambda_m}_{p_m, q_m(\cdot),\omega_m}$ to ${M{\mathop K\limits^.}}^{\alpha(\cdot),\lambda}_{p,q(\cdot),\omega}$ and  one of the following conditions holds:
\begin{enumerate}
\item[\rm (b1)] $q_{i+}=q_{i-}$, $C_0^{\alpha_i}$ and $C_\infty^{\alpha_i}\leq \alpha_i(0)-\alpha_{i\infty}$, for all $i=1,...,m$;
\item[\rm (b2)]$q_{i+}\neq q_{i-}$, $C_0^{\alpha_i} = C_\infty^{\alpha_i}=0$, $\lambda_i=\alpha_i(0)=\alpha_{i\infty}$, for all $i=1,...,m$;
\item[\rm (b3)] $q_{i+}\neq q_{i-}$, both $C_0^{\alpha_i}$ and $C_\infty^{\alpha_i}$ are less than $ \alpha_i(0)-\alpha_{i\infty}$, $C_\infty^{\alpha_i}+C_0^{\alpha_i}\leq C^{\alpha_i}$, $\lambda_i\in [\eta^0_i, \eta^1_i]\cap [\zeta^0_i,\zeta^1_i]$, for all $i=1,...,m$.
\end{enumerate}
Here $C^{\alpha_i}=\dfrac{q_{i-}(\alpha_i(0)-\alpha_{i\infty})(1+\frac{q_{i+}}{q_{i-}})}{q_{i+}}$ and $\eta^0_i, \eta^1_i, \zeta^0_i,\zeta^1_i$ are defined by
\[
\eta^0_i=\dfrac{C_0^{\alpha_i} \frac{q_{i-}}{q_{i+}}-\alpha_i(0)\frac{q_{i-}}{q_{i+}}+\alpha_{i\infty}}{1-\frac{q_{i-}}{q_{i+}}},\eta^1_i=\dfrac{C_0^{\alpha_i} \frac{q_{i+}}{q_{i-}}-\alpha_i(0)\frac{q_{i+}}{q_{i-}}+\alpha_{i\infty}}{1-\frac{q_{i+}}{q_{i-}}},
\]
\[
\zeta^0_i=\dfrac{C_\infty^{\alpha_i} \frac{q_{i+}}{q_{i-}}-\alpha_i(0)+\alpha_{i\infty}\frac{q_{i+}}{q_{i-}}}{\frac{q_{i+}}{q_{i-}}-1}, \zeta^1_i=\dfrac{C_\infty^{\alpha_i} \frac{q_{i-}}{q_{i+}}-\alpha_i(0)+\alpha_{i\infty}\frac{q_{i-}}{q_{i+}}}{\frac{q_{i-}}{q_{i+}}-1}.
\]
Then, we have that $\mathcal C_5^*$ is finite. Furthermore,
\[
\big\|H_{\Phi,\vec A}\big\|_{{M{\mathop K\limits^.}}^{\alpha_1(\cdot),\lambda_1}_{p_1, q_1(\cdot),\omega_1}\,\times\cdots\times {M{\mathop K\limits^.}}^{\alpha_m(\cdot),\lambda_m}_{p_m, q_m(\cdot),\omega_m}\to {M{\mathop K\limits^.}}^{\alpha(\cdot),\lambda}_{p,q(\cdot),\omega}}\gtrsim \mathcal C_5^*.
\]
\end{theorem}
\begin{proof}
Firstly, we prove for the case ${\rm (a)}$. From (\ref{DK1}), we call $\Theta_n$ is the greatest integer number such that $\rho_{\vec A} < 2^{-\Theta_n}$. Now, we replace $\Theta_n^*$ by $\Theta_n$ in the proof of Theorem \ref{TheoremVar-MorreyHerz}, and the other results are estimated in the same way. Then, by (\ref{nhungVarLeb2}), we get 
\begin{eqnarray}\label{MHVar1}
&&\big\|H_{\Phi,\vec A}(\vec f)\big\|_{M{\mathop K\limits^.}^{\alpha(\cdot), \lambda}_{p, q(\cdot),\omega}}\lesssim \Big(\int\limits_{\mathbb R^n}{\frac{\Phi(t)}{|t|^n}\prod\limits_{i=1}^{m}c_{A_i,q_i,\gamma_i}(t)}\big\|1\big\|_{L^{r_{1i}(t,\cdot)}}\times
\\
&&{\,\,\,\,\,\,\,\,\,\,\,\,\,\,\,\,\times\max\Big\{\big\|A_i(t)\big\|^{\lambda_i-\alpha_i(0)}, \big\|A_i(t)\big\|^{\lambda_i-\alpha_{i\infty}} \Big\}}dt\Big)\prod\limits_{i=1}^{m}\big\| f_i\big\|_{M{\mathop K\limits^.}^{\alpha_i(\cdot), \lambda_i}_{p_i, q_i(\cdot),\omega_i}}.
\nonumber
\end{eqnarray}
By the inequality (\ref{DK1.1}), we have
\begin{eqnarray}
&&\max\Big\{\big\|A_i(t)\big\|^{\lambda_i-\alpha_i(0)}, \big\|A_i(t)\big\|^{\lambda_i-\alpha_{i\infty}} \Big\}\nonumber
\\
&&\,\,\,\,\,\,\,\,\,\,\,\,\,\,\,\,\,\,\,\,\,\,\,\,\,\,\,\,\,\,\,\,\,\,\,\,\lesssim \big\|A_i^{-1}(t)\big\|^{-\lambda_i}\max\Big\{\big\|A_i^{-1}(t)\big\|^{\alpha_i(0)}, \big\|A_i^{-1}(t)\big\|^{\alpha_{i\infty}} \Big\}.
\nonumber
\end{eqnarray}
Thus, by (\ref{cAiq+-}), (\ref{MHVar1}) and $\mathcal C_3 <\infty$, we finish the proof for this case.
\vskip 5pt
Next, we will prove for case ${\rm (b)}$. By choosing
\[
f_i(x)=|x|^{-\alpha_i(x)-\frac{n}{q_i(x)}-\gamma_i+\lambda_i}, 
\]
it is evident  that $\big\| f_i\big\|_{{M{\mathop K\limits^.}}^{\alpha_i(\cdot),\lambda_i}_{p_i, q_i(\cdot),\omega_i}}>0$, for all $i=1,...,m$. Now, we need to show that 
\begin{equation}\label{fiMHVar}
 \big\| f_i\big\|_{{M{\mathop K\limits^.}}^{\alpha_i(\cdot),\lambda_i}_{p_i, q_i(\cdot),\omega_i}}<\infty,\,\textit{\rm for all }\, i=1,...,m.
\end{equation}
Indeed, we find
\[
F_{q_i}(f_i\omega_i.\chi_k)=\int\limits_{C_k}{|x|^{(\lambda_i-\alpha_i(x))q_i(x)-n}dx}=\int\limits_{2^{k-1}}^{2^k}\int\limits_{S^{n-1}}{r^{(\lambda_i-\alpha_i(r.x'))q_i(r.x')-1}d\sigma(x')}dr.
\]
\textit{Case 1}: $k\leq 0$. Since $\alpha_i \in \mathbf C_\infty^{\rm log}(\mathbb R^n)$, it follows that
\[
- C_{\infty}^{\alpha_i}+ \alpha_{i\infty}\leq \alpha_i(x)\leq \alpha_{i\infty}+ C_{\infty}^{\alpha_i}.
\]
As a consequence, we get
\begin{eqnarray}
F_{q_i}(f_i\omega_i.\chi_k)&\leq& \int\limits_{2^{k-1}}^{2^k}\int\limits_{S^{n-1}}{r^{(\lambda_i-\alpha_{i\infty}-C_{\infty}^{\alpha_i})q_i(r.x')-1}d\sigma(x')}dr \nonumber
\\
&\lesssim &\textit{\rm max}\big\{ 2^{k(\lambda_i-\alpha_{i\infty}-C_{\infty}^{\alpha_i})q_{i-}}, 2^{k(\lambda_i-\alpha_{i\infty}-C_{\infty}^{\alpha_i})q_{i+}} \big\}\nonumber.
\end{eqnarray}
Thus, by (\ref{maxminvar}), we obtain
\begin{eqnarray}\label{fiXkVar}
\big\|f_i\chi_k\big\|_{L^{q_i(\cdot)}_{\omega_i}}&\lesssim & \textit{\rm max}\big\{ 2^{k(\lambda_i-\alpha_{i\infty}-C_{\infty}^{\alpha_i})\frac{q_{i-}}{q_{i+}}}, 2^{k(\lambda_i-\alpha_{i\infty}-C_{\infty}^{\alpha_i})\frac{q_{i+}}{q_{i-}}} \big\}
\nonumber
\\
&=& 2^{k(\lambda_i-\alpha_{i\infty}-C_{\infty}^{\alpha_i})\beta_{i\infty}},
\end{eqnarray}
where 
\[
\beta_{i\infty} = \left\{ \begin{array}{l}
\dfrac{q_{i+}}{q_{i-}}, \, \textit{\rm if}\, \lambda_i-\alpha_{i\infty}-C_{\infty}^{\alpha_i} <0,
\\
\\
\dfrac{q_{i-}}{q_{i+}}, \,\textit{\rm otherwise.}
\end{array} \right.
\]

\textit{Case 2}: $k>0$. 
Since $\alpha_i \in \mathbf C_0^{\rm log}(\mathbb R^n)$, we have 
\begin{equation}\label{C0alphai}
- C_{0}^{\alpha_i}+ \alpha_i(0)\leq \alpha_i(x)\leq \alpha_i(0)+ C_{0}^{\alpha_i}.
\end{equation}
Denote 
\[
\beta_{i0} = \left\{ \begin{array}{l}
\dfrac{q_{i+}}{q_{i-}}, \, \textit{\rm if }\, \lambda_i-\alpha_{i}(0)+C_{0}^{\alpha_i} \geq 0,
\\
\\
\dfrac{q_{i-}}{q_{i+}}, \,\textit{\rm otherwise.}
\end{array} \right.
\]
By having (\ref{C0alphai}) and estimating in the same way as the case 1, we deduce
\begin{equation}\label{fiXkVar1}
\big\|f_i\chi_k\big\|_{L^{q_i(\cdot)}_{\omega_i}} \lesssim 2^{k(\lambda_i-\alpha_{i}(0)+C_{0}^{\alpha_i})\beta_{i0}}.
\end{equation}
Next, it follows from Proposition 2.5 in \cite{LZ2014} that
\begin{equation}\label{ptfiMHVar}
\big\|f_i\big\|_{M{\mathop K\limits^.}^{\alpha_i(\cdot), \lambda}_{p_i,q_i(\cdot),\omega_i}} \lesssim \max\big\{
\sup\limits_{k_0<0,k_0\in\mathbb Z}E_{i,1}, \sup\limits_{k_0\geq 0,k_0\in\mathbb Z}(E_{i,2}+E_{i,3})
\big\},
\end{equation}
where
\begin{eqnarray}
&&E_{i,1}=2^{-k_0\lambda_i}\Big(\sum\limits_{k=-\infty}^{k_0}2^{k\alpha_i(0)p_i}\big\|f_i\big\|^{p_i}_{L^{q_i(\cdot)}_{\omega_i}}\Big)^{\frac{1}{p_i}},\nonumber
\\
&&E_{i,2} = 2^{-k_0\lambda_i}\Big(\sum\limits_{k=-\infty}^{-1}2^{k\alpha_i(0)p_i}\big\|f_i\big\|^{p_i}_{L^{q_i(\cdot)}_{\omega_i}}\Big)^{\frac{1}{p_i}},\nonumber
\\
&&E_{i,3}=2^{-k_0\lambda_i}\Big(\sum\limits_{k=0}^{k_0}2^{k\alpha_{i\infty} p_i}\big\|f_i\big\|^{p_i}_{L^{q_i(\cdot)}_{\omega_i}}\Big)^{\frac{1}{p_i}}\nonumber.
\end{eqnarray}
Notice that the relation $\alpha_i(0)+(\lambda_i-\alpha_{i\infty}-C_{\infty}^{\alpha_i})\beta_{i\infty}$ is required positively which is proved later, beacause $E_{i,1}$ is infinite otherwise. Thus, because of  (\ref{fiXkVar}), we have $E_{i,1}$ and $E_{i,2}$ are dominated by 
\begin{eqnarray}\label{E12iMHVar}
E_{i,1}&\lesssim& 2^{-k_0\lambda_i}\Big(\sum\limits_{k=-\infty}^{k_0}2^{k\alpha_i(0)p_i}2^{kp_i(\lambda_i-\alpha_{i\infty}-C_{\infty}^{\alpha_i})\beta_{i\infty}}\Big)^{\frac{1}{p_i}}
\nonumber
\\
&\lesssim& 2^{k_0(\alpha_i(0)+ (\lambda_i-\alpha_{i\infty}-C_{\infty}^{\alpha_i})\beta_{i\infty}-\lambda_i)},
\nonumber
\\
E_{i,2}&\lesssim& 2^{-k_0\lambda_i}.2^{-\alpha_i(0)-(\lambda_i-\alpha_{i\infty}-C_{\infty}^{\alpha_i})\beta_{i\infty}}\lesssim 2^{-k_0\lambda_i}.
\end{eqnarray}
\\
By (\ref{fiXkVar1}), we have $E_{i,3}$ is controlled by
\begin{eqnarray}
E_{i,3}&\lesssim & 2^{-k_0\lambda_i}+2^{-k_0\lambda_i}\Big(\sum\limits_{k=1}^{k_0}2^{k\alpha_{i\infty} p_i}\big\|f_i\big\|^{p_i}_{L^{q_i(\cdot)}_{\omega_i}}\Big)^{\frac{1}{p_i}}
\nonumber
\\
&\lesssim & 2^{-k_0\lambda_i}+2^{-k_0\lambda_i}\Big(\sum\limits_{k=1}^{k_0}2^{kp_i(\alpha_{i\infty}+(\lambda_i-\alpha_i(0)+C_0^{\alpha_i})\beta_{i0})}\Big)^{\frac{1}{p_i}}
\nonumber
\\
&\lesssim & \left\{ \begin{array}{l}
2^{-k_0\lambda_i}(k_0^{\frac{1}{p}}+1), \,\textit{\rm if}\,\, \alpha_{i\infty}+(\lambda_i-\alpha_i(0)+C_0^{\alpha_i})\beta_{i0} =0, 
\\
\\
2^{-k_0\lambda_i}+2^{-k_0(\lambda_i-\alpha_{i\infty}-(\lambda_i-\alpha_i(0)+C_0^{\alpha_i})\beta_{i0})},\,\textit{\rm otherwise}.
\end{array} \right.
\nonumber
\end{eqnarray}
This implies that
\begin{equation}\label{Ei3MHVar}
E_{i,3} \lesssim  2^{-k_0\lambda_i}(k_0^{\frac{1}{p}}+1)+ 2^{-k_0(\lambda_i-\alpha_{i\infty}-(\lambda_i-\alpha_i(0)+C_0^{\alpha_i})\beta_{i0})}.
\end{equation}
For convenience, we set
\begin{equation}\label{DKfi}
\left\{ \begin{array}{l}
\theta_{i0}= \lambda_i-\alpha_{i\infty}-(\lambda_i-\alpha_i(0)+C_0^{\alpha_i})\beta_{i0},
\\
\theta_{i\infty}= \alpha_i(0)+ (\lambda_i-\alpha_{i\infty}-C_{\infty}^{\alpha_i})\beta_{i\infty}-\lambda_i.
\end{array} \right.
\end{equation}
Combining (\ref{ptfiMHVar})-(\ref{DKfi}), we get that
\[
\big\|f_i\big\|_{M{\mathop K\limits^.}^{\alpha_i(\cdot), \lambda}_{p_i,q_i(\cdot)}(\omega_i)} \lesssim \max\big\{
\sup\limits_{k_0<0,k_0\in\mathbb Z}2^{k_0\theta_{i\infty}}, \sup\limits_{k_0\geq 0,k_0\in\mathbb Z}\big(2^{-k_0\lambda_i}(k_0^{\frac{1}{p}}+1)+2^{-k_0\theta_{i0}}\big)
\big\}.
\]
From the above estimation, we will finish the proof of (\ref{fiMHVar}) if  the following result can be proved
\begin{equation}\label{DKfi*}
\theta_{i0}\geq 0\,\textit{\rm and}\,\,\theta_{i\infty}\geq 0.
\end{equation}
In order to do this, let us consider three cases as follows.
\vskip 5pt
\textit{Case b1}. By $q_{i+}=q_{i-}$, we have $\beta_{i0}=\beta_{i\infty}=1$. So, by the information
of $C_0^{\alpha_i}$ and $C_\infty^{\alpha_i}$, it is easy to have the desired result (\ref{DKfi*}).
\vskip 5pt
\textit{Case b2}. In this case, we find $\theta_{i0}=\theta_{i\infty}=0$. This follows immediately that the result (\ref{DKfi*}) is true.
\vskip 5pt
\textit{Case b3}. Because both $C_0^{\alpha_i}$ and $C_\infty^{\alpha_i}$ are less than $ \alpha_i(0)-\alpha_{i\infty}$, we have
$[\eta^0_i, \eta^1_i]$ and $[\zeta^0_i,\zeta^1_i]$ are not empty sets. Also, we obtain
\begin{equation}\label{empty}
\alpha_i(0)-C_{0}^{\alpha_i}\in [\eta^0_i, \eta^1_i]\,\textit{\rm and}\,\,\alpha_{i\infty}+C_\infty^{\alpha_i}\in [\zeta^0_i,\zeta^1_i].
\end{equation}
From $C_\infty^{\alpha_i}+C_0^{\alpha_i}\leq C^{\alpha_i}$, it implies that $\eta^1_i\geq \zeta^0_i\,\textit{\rm and}\,\,\zeta^1_i\geq \eta^0_i.$
Hence, we also have $[\eta^0_i, \eta^1_i]\cap [\zeta^0_i,\zeta^1_i]$ is not an empty set. Thus, by (\ref{empty}), we observe that
\begin{eqnarray}
[\eta^0_i, \eta^1_i]\cap [\zeta^0_i,\zeta^1_i]
&=&\Big([\eta^0_i,\alpha_i(0)-C_{0}^{\alpha_i})\cup[\alpha_i(0)-C_{0}^{\alpha_i}, \eta^1_i]\Big)\nonumber
\\
&&\,\,\,\,\cap \Big([\zeta^0_i,\alpha_{i\infty}+C_\infty^{\alpha_i})\cup[\alpha_{i\infty}+C_\infty^{\alpha_i},\zeta^1_i]\Big).
\nonumber
\end{eqnarray}
For the above separation, by calculating directly and defining $\beta_{i0}$, $\beta_{i\infty}$, we have that
\[
\lambda_i\in [\eta^0_i, \eta^1_i]\cap [\zeta^0_i,\zeta^1_i]\Leftrightarrow (\ref{DKfi*})\,\,\, \text {\rm holds}.
\]
This claims that the desired estimation (\ref{fiMHVar}) is completed.
Combining (\ref{DK1.2}) and (\ref{C0alphai}), we obtain
\begin{eqnarray}
H_{\Phi,\vec A}(\vec f)(x)&=&\int\limits_{\mathbb R^n}{\dfrac{\Phi(t)}{|t|^n}\prod\limits_{1}^{m}{|A_i(t)x|^{-\alpha_i(x)-\frac{n}{q_i(x)}-\gamma_i+\lambda_i}}dt}\nonumber
\\
&\gtrsim & \Big(\int\limits_{\mathbb R^n}{\dfrac{\Phi(t)}{|t|^n}\prod\limits_{1}^{m}{\big\|A_i^{-1}(t)\big\|^{\alpha_i(x)+\frac{n}{q_i(x)}+\gamma_i-\lambda_i}}dt}\Big).|x|^{-\alpha(x)-\frac{n}{q(x)}-\gamma+\lambda}\nonumber
\\
&\gtrsim & \mathcal C_5^*.|x|^{-\alpha(x)-\frac{n}{q(x)}-\gamma+\lambda}\nonumber.\
\end{eqnarray}
From this, because of (\ref{fiMHVar}) and assuming that $H_{\Phi,\vec{A}}$ is a bounded operator, 
we conclude 
\[
\big\|H_{\Phi,\vec{A}}\big\|_{{M{\mathop K\limits^.}}^{\alpha_1(\cdot),\lambda_1}_{p_1, q_1(\cdot),\omega_1}\,\times\cdots\times {M{\mathop K\limits^.}}^{\alpha_m(\cdot),\lambda_m}_{p_m, q_m(\cdot),\omega_m}\to {M{\mathop K\limits^.}}^{\alpha(\cdot),\lambda}_{p,q(\cdot),\omega}}\gtrsim\mathcal C_5^*\frac{\big\||\cdot|^{-\alpha(\cdot)-\frac{n}{q(\cdot)}-\gamma+\lambda}\big\|_{{M{\mathop K\limits^.}}^{\alpha(\cdot),\lambda}_{p,q(\cdot),\omega}}}{\prod\limits_{i=1}^m \big\|f_i\big\|_{{M{\mathop K\limits^.}}^{\alpha_i(\cdot),\lambda_i}_{p_i, q_i(\cdot),\omega_i}}} .
\]
This implies the desired assertion.
\end{proof}
\begin{theorem}\label{TheoremVarHerz1}
Suppose that the assumptions of Theorem \ref{TheoremVarHerz} and the hypothesis (\ref{DKVarLeb}) in Theorem \ref{TheoremVarLebesgue1} are true.
\\
{\rm (a)} If 
\[
\mathcal C_6 =\int\limits_{\mathbb R^n}{\dfrac{\Phi(t)}{|t|^n}\prod\limits_{i=1}^{m}\max\Big\{\big\|A_i^{-1}(t)\big\|^{\frac{n}{q_{i+}}+\gamma_i}, \big\|A_i^{-1}(t)\big\|^{\frac{n}{q_{i-}}+\gamma_i}\Big\}\big\|A_i^{-1}(t)\big\|^{\alpha_{i}(0)}}\big\|1\big\|_{L^{r_{1i}(t,\cdot)}}dt<\infty,
\]
then $H_{\Phi,\vec{A}}$ is a bounded operator from ${\mathop K\limits^.}_{q_1(\cdot),{\omega_1}}^{{\alpha _1(\cdot)}, p_1} \times \cdots\times {\mathop K\limits^.}_{q_m(\cdot),{\omega_m}}^{{\alpha _m(\cdot)}, p_m} $ to ${\mathop K\limits^.}_{q(\cdot),{\omega}}^{{\alpha(\cdot)}, p}.$
\\
\\
{\rm (b)} Denote by
\[
\mathcal C_6^* = \left\{ \begin{array}{l}
\int\limits_{\mathbb R^n}{\dfrac{\Phi(t)}{|t|^n}\prod\limits_{i=1}^{m}\big\|A_i^{-1}(t)\big\|^{\alpha_i(0)+\frac{n}{q_{i}}+\gamma_i}}dt,\,\textit{\rm if}\,\,q_{i+}=q_{i-} \,\,\textit{\rm for all}\,\,i=1,...,m,
\\
\\
\int\limits_{\mathbb R^n}{\dfrac{\Phi(t)}{|t|^n}\prod\limits_{i=1}^{m}\min\Big\{\big\|A_i^{-1}(t)\big\|^{\frac{n}{q_{i+}}+\gamma_i}, \big\|A_i^{-1}(t)\big\|^{\frac{n}{q_{i-}}+\gamma_i}\Big\}\big\|A_i^{-1}(t)\big\|^{\|\alpha_i\|_{L^{\infty}}}}dt,\,\textit{\rm otherwise.}
\end{array} \right.\]
Let $H_{\Phi,\vec{A}}$ be a bounded operator from ${\mathop K\limits^.}_{q_1(\cdot),{\omega_1}}^{{\alpha _1(\cdot)}, p_1} \times \cdots\times {\mathop K\limits^.}_{q_m(\cdot),{\omega_m}}^{{\alpha _m(\cdot)}, p_m} $ to ${\mathop K\limits^.}_{q(\cdot),{\omega}}^{{\alpha(\cdot)}, p}$ and one of the following conditions is satisfied:
\begin{enumerate}
\item[\rm (b1)] $q_{i-}=q_{i+},\,\textit{\rm for all}\,\,i=1,...,m$;
\item[\rm (b2)] The case $\rm (b1)$ is not true and $\alpha_i(0) < \|\alpha_i\|_{L^\infty}\dfrac{q_{i-}}{q_{i+}},\,\textit{\rm for all}\,\,i=1,...,m$.
\end{enumerate}
Then, we have that $\mathcal C_6^*$ is finite. Furthermore, there exists $C>0$ such that the operator norm of $H_{\Phi,\vec A}$ is not greater than $C.\mathcal C_6^*$.
\end{theorem}
\begin{proof}
In the case $\rm (a)$, by combining Theorem \ref{TheoremVarHerz} and the part $\rm (a)$ of Theorem \ref{TheoremVarMorreyHerz1}, we immediately imply the desired result. 
\vskip 5pt
In the case $\rm (b1)$, we have that $q_1(\cdot), ... , q_m(\cdot),$ and $ q(\cdot)$ are constant. Thus, for all $i=1,..., m$,
we will choose the function $f_i$ as follows:
\[{f_i}(x) =
\begin{cases}
0,\,\,\;\;\;\;\;\,\,\,\,\,\,\,\,\,\,\,\,\,\,\,\,\,\,\,\,\,\,\,\,\,\,\,{\rm if }\, | x | < p_{\vec A}^{ - 1},&\\
  {| x |^{ - {\alpha _i(0)} - \frac{n}{q_i}-\gamma_i - \varepsilon }},\,\,{\rm otherwise.}&
\end{cases}
\]
It is obvious to see that when $k$ is an integer number satisfying $k\leq\frac{-{\rm log}(\rho_{\vec A})}{{\rm log}(2)}$ then $\big\|f_i\chi_k\big\|_{L^{q_i}_{\omega_i}} =0$. Otherwise, we have
\[
\big\|f_i\chi_k\big\|_{L^{q_i}_{\omega_i}}^{q_i}\lesssim 2^{-kq_i(\alpha_i(0)+\varepsilon)}\dfrac{\big(2^{q_i(\alpha_i(0)+\varepsilon)}-1\big)}{{q_i(\alpha_i(0)+\varepsilon)}}.
\]
Hence, by applying Proposition 3.8 in \cite{Almeida2012} again and $\alpha_i(0)=\alpha_{i\infty}$, we find
\begin{eqnarray}
{\left\| f_i\right\|_{{\mathop K\limits^.}_{q_i, {\omega_i}}^{{\alpha_i(\cdot)}, p_i}}} &\lesssim & {\left\{ {\sum\limits_{k =\rho}^\infty  {{2^{k{\alpha _i(0)}{p_i}}}} \left\| {{f_i}{\chi _k}} \right\|_{L^{q_i}_{\omega_i}}^{{p_i}}} \right\}^{\frac{1}{{p_i}}}}\nonumber
\\
&\lesssim& {\left(\dfrac{{{2^{{q_i}({\alpha _i(0)} + \varepsilon )}} - 1}}{q_i(\alpha_i(0)+\varepsilon)} \right)^{\frac{1}{{{q_i}}}}}{\left( {\frac{{{2^{\varepsilon {p_i} - \rho\varepsilon {p_i}}}}}{{{2^{\varepsilon {p_i}}} - 1}}} \right)^{\frac{1}{{{p_i}}}}} < \infty,
 \nonumber
\end{eqnarray}
where $\rho$ is the smallest integer number such that $\rho>\frac{-{\rm log}(\rho_{\vec A})}{{\rm log}(2)}$.
Estimating as (\ref{HAf3Leb}), we have
\begin{equation}\label{HAf3}
{H_{\Phi ,\vec A }}(\vec{f)} (x) \gtrsim \Big(\int\limits_{U} {\frac{{\Phi (t)}}{{{{\left| t \right|}^n}}}} \prod\limits_{i = 1}^m {{{\big\| A_i^{-1}(t)\big\|}^{\alpha_i(0)+\frac{n}{q_i}+\gamma_i+ \varepsilon }}} dt\Big)|x|^{-\alpha(0)- \frac{{n }}{{{q}}}-\gamma - m\varepsilon}\chi_{\mathbb R^n\setminus B(0,\varepsilon^{-1})}(x).
\end{equation}
Let $k_0$ be the smallest integer number such that $2^{k_0-1}\geq \varepsilon^{-1}$. Using Proposition 3.8 in \cite{Almeida2012} again, $\alpha_i(0)=\alpha_{i\infty}$ and (\ref{HAf3}), we obtain 
\begin{eqnarray}\label{HAf4}
\big\| {{H_{\Phi ,\vec A }}\big( {\vec f } \big)} \big\|_{{\mathop K\limits^.}_{q, {\omega}}^{{\alpha(\cdot)}, p}}^p&\gtrsim& \sum\limits_{k = {k_0}}^\infty  {{2^{k\alpha(0) p}}} {\Big( {\int\limits_{{2^{k - 1}} < \left| x \right| \leq {2^k}} {{{\left| x \right|}^{ - \varepsilon mq - \alpha(0)q - n}}} }dx \Big)^{\frac{p}{q}}}\times
\nonumber
\\
&&\,\,\times{\Big( {\int\limits_U {\frac{{\Phi (t)}}{{{{\left| t \right|}^n}}}\prod\limits_{i = 1}^m {{{\left\| {{A_i}^{ - 1}(t)} \right\|}^{{\alpha_i(0)} + \frac{{n}}{{{q_i}}}+\gamma_i + \varepsilon }}dt} } } \Big)^p}.
\end{eqnarray}
An elementary calculation leads that
\begin{equation}\label{Vanh}
\sum\limits_{k = {k_0}}^\infty  {{2^{k\alpha(0)p}}} {\Big( {\int\limits_{{2^{k - 1}} < \left| x \right| \leq {2^k}} {{{\left| x \right|}^{ - \varepsilon mq - \alpha(0)q - n}}} dx} \Big)^{\frac{p}{q}}}\gtrsim\Big( {\frac{{{2^{ - {k_0}\varepsilon mp}}}}{{1 - {2^{ - \varepsilon mp}}}}} \Big){\Big( {\frac{{{2^{q(\varepsilon m + \alpha(0))}} - 1}}{{q(\varepsilon m + \alpha(0))}}} \Big)^{\frac{p}{q}}}.
\end{equation}
For simplicity of notation, we write
\[
\vartheta^* \big(\varepsilon\big) = \dfrac{{\left( {\frac{{{2^{ - {k_0}\varepsilon mp}}}}{{1 - {2^{ - \varepsilon mp}}}}} \right)^{\frac{1}{p}}{{\left( {\frac{{{2^{q(\varepsilon m + \alpha(0))}} - 1}}{{q(\varepsilon m + \alpha(0))}}} \right)}^{\frac{1}{q}}}}}{{\prod\limits_{i = 1}^m {{{\left( {\frac{{{2^{\varepsilon {p_i} - \rho \varepsilon {p_i}}}}}{{{2^{\varepsilon {p_i}}} - 1}}} \right)}^{\frac{1}{{{p_i}}}}}{{\left( {\frac{{{2^{{q_i}({\alpha _i(0)} + \varepsilon )}} - 1}}{{{q_i}({\alpha _i(0)} + \varepsilon )}}} \right)}^{\frac{1}{{{p_i}}}}}} }}.
\]
Therefore, by (\ref{HAf4}) and (\ref{Vanh}), we estimate
\begin{eqnarray}\label{HAf5}
\big\| {{H_{\Phi ,\vec A }}\big( {\vec f } \big)} \big\|_{{\mathop K\limits^.}_{q,\omega}^{{\alpha(\cdot)}, p}}&\gtrsim& \varepsilon^{-m\varepsilon}\vartheta^*.\prod\limits_{i=1}^m\big\|f_i\big\|_{{\mathop K\limits^.}_{q_i,{\omega_i}}^{{\alpha_i(\cdot)}, p_i}}\times
\\
&&\times{\Big( {\int\limits_U {\frac{{\Phi (t)}}{{{{\left| t \right|}^n}}}\prod\limits_{i = 1}^m {{{\left\| {{A_i}^{ - 1}(t)} \right\|}^{{\alpha_i(0)} + \frac{n}{q_i}+\gamma_i}}\prod\limits_{i=1}^m\big\|A_i^{-1}(t)\big\|^\varepsilon\varepsilon^{m\varepsilon} dt} } } \Big)}.
\nonumber
\end{eqnarray}
By (\ref{lienhieppi}), it is easy to show that
\[
\mathop {\lim }\limits_{\varepsilon  \to {0^ + }} \varepsilon^{-m\varepsilon}\vartheta^*(\varepsilon) = a > 0.
\]
Thus, by (\ref{HAf5}), (\ref{prodLeb}) and the dominated convergence theorem of Lebesgue, we complete the proof for this case.
\vskip 5pt
Next, let us consider the case $\rm (b2)$. We now choose the functions $f_i$ for all $i=1,...,m$ as follows:
\[{f_i}(x) =
\begin{cases}
0,\,\,\;\;\;\;\;\,\,\,\,\,\,\,\,\,\,\,\,\,\,\,\,\,\,\,\,\,\,\,\,\,\,\,\,\,\,\,\,\,\,\,\,\,{\rm if }\,\, | x | < \rho_{\vec A}^{ - 1},&\\
  {|x|^{ - {\|\alpha _i\|_{L^{\infty}}} - \frac{n}{q_i(x)}-\gamma_i - \varepsilon }},\,\,{\rm otherwise.}&
\end{cases}
\]
Thus, we have
\[
F_{q_i}(f_i\omega_i.\chi_k)= \int\limits_{C^k}{|x|^{-(\|\alpha_i\|_{L^\infty}+\varepsilon)q_i(x)-n}dx}=\int\limits_{2^{k-1}}^{2^k}\int\limits_{S^{n-1}}{r^{-(\|\alpha_i\|_{L^\infty}+\varepsilon)q_i(r.x')-1}d\sigma(x')}dr.
\]
Hence, by letting $k\leq 0$, $F_{q_i}(f_i\omega_i.\chi_k)$ is controlled as follows
\begin{eqnarray}
\int\limits_{2^{k-1}}^{2^k}\int\limits_{S^{n-1}}{r^{-(\|\alpha_i\|_{L^\infty}+\varepsilon)q_{i+}-1}d\sigma(x')}dr \lesssim 2^{-k(\|\alpha_i\|+\varepsilon)q_{i+}}. \dfrac{2^{(\|\alpha_i\|_{L^{\infty}}+\varepsilon)q_{i+}}-1}{q_{i+}(\|\alpha_i\|_{L^\infty}+\varepsilon)}.
\nonumber
\end{eqnarray}
As a consequence of the above estimate, by (\ref{maxminvar}), we get 
\begin{equation}\label{fichikVarHerz}
\big\|f_i\chi_k\big\|_{L^{q_i(\cdot)}_{\omega_i}}\lesssim (\eta_{j+})^{\frac{1}{q_{i-}}}2^{-k(\|\alpha_i\|_{L^\infty}+\varepsilon)\frac{q_{i+}}{q_{i-}}},
\end{equation}
where $\eta_{j+}=\dfrac{2^{(\|\alpha_i\|_{L^{\infty}}+\varepsilon)q_{i+}}-1}{q_{i+}(\|\alpha_i\|_{L^\infty}+\varepsilon)}$. Otherwise, by the similar argument as above, we also obtain
\begin{equation}\label{fichikVarHerz1}
\big\|f_i\chi_k\big\|_{L^{q_i(\cdot)}_{\omega_i}}\lesssim (\eta_{j-})^{\frac{1}{q_{i-}}}2^{-k(\|\alpha_i\|+\varepsilon)\frac{q_{i-}}{q_{i+}}},
\end{equation}
where $\eta_{j-}=\dfrac{2^{(\|\alpha_i\|_{L^{\infty}}+\varepsilon)q_{i-}}-1}{q_{i-}(\|\alpha_i\|_{L^\infty}+\varepsilon)}$. From defining $\rho$ and assuming $\alpha_i(0)=\alpha_{i\infty}$, by Proposition 3.8 in \cite{Almeida2012} again, we get
\begin{eqnarray}\label{fiVarHerz}
\big\| f_i\big\|_{{\mathop K\limits^.}_{q_i(\cdot),\omega_i}^{{\alpha _i(\cdot)}, p_i}} &\leq & {\left\{ {\sum\limits_{k =\rho}^0 {{2^{k{\alpha _i(0)}{p_i}}}} \left\| {{f_i}{\chi _k}} \right\|_{L^{q_{i}(\cdot)}_{\omega_i}}^{{p_i}}} \right\}^{\frac{1}{{p_i}}}}
\\
&&\,\,\,\,\,\,\,\,\,\,\,\,\,+{\left\{ {\sum\limits_{k =1}^\infty  {{2^{k{\alpha _i(0)}{p_i}}}} \left\| {{f_i}{\chi _k}} \right\|_{L^{q_{i}(\cdot)}_{\omega_i}}^{{p_i}}} \right\}^{\frac{1}{{p_i}}}}.
\nonumber
\nonumber
\end{eqnarray}
Notice that, from assuming in this case, we deduce
\[
\alpha_i(0)-(\|\alpha_i\|_{L^\infty}+\varepsilon)\dfrac{q_{i-}}{q_{i+}}<0,\,\textit{\rm for all}\,\,\varepsilon\in \mathbb R^+.
\]
Thus, by (\ref{fichikVarHerz})-(\ref{fiVarHerz}), $\big\| f_i\big\|_{{\mathop K\limits^.}_{q_i(\cdot),\omega_i}^{{\alpha _i(\cdot)}, p_i}}$ is dominated by
\[
\eta_{j+}^{\frac{1}{q_{i-}}}\Big(\dfrac{2^{(-\rho+1)p_i(-\alpha_i(0)+(\|\alpha_i\|_{L^\infty}+\varepsilon)\frac{q_{i+}}{q_{i-}}}-1}{2^{p_i(-\alpha_i(0)+(\|\alpha_i\|_{L^\infty}+\varepsilon)\frac{q_{i+}}{q_{i-}})}-1}\Big)^{\frac{1}{p_i}}+ \eta_{j-}^{\frac{1}{q_{i-}}}\Big(\dfrac{2^{p_i(\alpha_i(0)-(\|\alpha_i\|_{L^\infty}+\varepsilon)\frac{q_{i-}}{q_{i+}})}}{1-2^{p_i(\alpha_i(0)-(\|\alpha_i\|_{L^\infty}+\varepsilon)\frac{q_{i-}}{q_{i+}})}}\Big)^{\frac{1}{p_i}}.
\]
This implies that
\begin{equation}\label{bdtfiVarHez}
\big\| f_i\big\|_{{\mathop K\limits^.}_{q_i(\cdot),\omega_i}^{{\alpha _i(\cdot)}, p_i}}\lesssim
\dfrac{I_i(\varepsilon)}{\Big(1-2^{p_i(\alpha_i(0)-(\|\alpha_i\|_{L^\infty}+\varepsilon)\frac{q_{i-}}{q_{i+}})}\Big)^{\frac{1}{p_i}}},
\end{equation}
where
\begin{eqnarray}
I_i(\varepsilon)&=& \eta_{j+}^{\frac{1}{q_{i-}}}\Big(\dfrac{2^{(-\rho+1)p_i(-\alpha_i(0)+(\|\alpha_i\|_{L^\infty}+\varepsilon)\frac{q_{i+}}{q_{i-}}}-1}{2^{p_i(-\alpha_i(0)+(\|\alpha_i\|_{L^\infty}+\varepsilon)\frac{q_{i+}}{q_{i-}})}-1}\Big)^{\frac{1}{p_i}}.\Big(1-2^{p_i(\alpha_i(0)-(\|\alpha_i\|_{L^\infty}+\varepsilon)\frac{q_{i-}}{q_{i+}})}\Big)^{\frac{1}{p_i}}
\nonumber
\\
&&+\,\eta_{j-}^{\frac{1}{q_{i-}}}{2^{\alpha_i(0)-(\|\alpha_i\|_{L^\infty}+\varepsilon)\frac{q_{i-}}{q_{i+}}}}.
\nonumber
\end{eqnarray}
On the other hand, by the similar estimating as (\ref{HAf3Leb}), we also obtain
\[
H_{\Phi,\vec A}(\vec f)(x)\geq \Big(\int\limits_{U} {\frac{{\Phi (t)}}{{{{\left| t \right|}^n}}}} \prod\limits_{i = 1}^m {{{| {{A_i}(t)x}|}^{ - {\|\alpha _i\|_{L^\infty}} - \frac{n}{q_i(x)}-\gamma_i - \varepsilon }}} dt\Big)\chi_{\mathbb R^n\setminus B(0,\varepsilon^{-1})}(x).
\]
For convenience, we put
\[
\Gamma^*_\varepsilon= \int\limits_{U}{\dfrac{\Phi(t)}{|t|^n}\prod\limits_{i=1}^{m}\min\Big\{\big\|A_i^{-1}(t)\big\|^{\frac{n}{q_{i+}}+\gamma_i}, \big\|A_i^{-1}(t)\big\|^{\frac{n}{q_{i-}}+\gamma_i}\Big\}\big\|A_i^{-1}(t)\big\|^{\|\alpha_i\|_{L^{\infty}}+\varepsilon}}dt.
\]
From this, by (\ref{DK1.2}), it is not hard to see that
\begin{equation}
H_{\Phi,\vec A}(\vec f)(x)\geq \Gamma^*_\varepsilon.|x|^{-(\sum\limits_{i=1}^{m}\|\alpha_i\|_{L^\infty})-\frac{n}{q(x)}-\gamma-m\varepsilon}\chi_{\mathbb R^n\setminus B(0,\varepsilon^{-1})}=: \Gamma^*_\varepsilon.g(x),
\nonumber
\end{equation}
where we denote $g(x)=|x|^{-(\sum\limits_{i=1}^{m}\|\alpha_i\|_{L^\infty})-\frac{n}{q(x)}-\gamma-m\varepsilon}\chi_{\mathbb R^n\setminus B(0,\varepsilon^{-1})}$.
\\
Since $\alpha(0)=\alpha_{\infty}$, we deduce that
\begin{equation}\label{HfVarHez}
\big\|H_{\Phi,\vec A}(\vec f)\big\|_{{\mathop K\limits^.}_{q(\cdot),\omega}^{{\alpha(\cdot)}, p}}\geq \Gamma^*_\varepsilon.\Big(\sum\limits_{k=k_0}^{\infty}2^{k\alpha(0)p}\big\|g\chi_k\big\|_{L^{q(\cdot)}_{\omega}}^p\Big)^{\frac{1}{p}},
\end{equation}
by using Proposition 3.8 in \cite{Almeida2012} again. Here we recall that $k_0$ is the smallest integer number so that $2^{k_0-1}\geq \varepsilon^{-1}$. 
Let us now show that
\begin{equation}\label{gvar}
\Big(\sum\limits_{k=k_0}^{\infty}2^{k\alpha(0)p}\big\|g\chi_k\big\|_{L^{q(\cdot)}_{\omega}}^p\Big)^{\frac{1}{p}}\gtrsim \eta_{+}^{\frac{1}{q_+}}\dfrac{2^{k_0(\alpha(0)-(\sum\limits_{i=1}^{m}\|\alpha_i\|_{L^\infty}+\varepsilon m)\frac{q_+}{q_-})}}{\Big(1-2^{\alpha(0)-(\sum\limits_{i=1}^{m}\|\alpha_i\|_{L^\infty}+\varepsilon m)\frac{q_+}{q_-}}\Big)^{\frac{1}{p}}},
\end{equation}
where $\eta_{+}=\dfrac{2^{(\sum\limits_{i=1}^{m}\|\alpha_i\|_{L^\infty}+\varepsilon m)q_+}-1}{(\sum\limits_{i=1}^{m}\|\alpha_i\|_{L^\infty}+\varepsilon m)q_+}.$ Indeed, by $k\geq k_0>1$, we get
\begin{eqnarray}
F_{q}(g\omega.\chi_k)&=& \int\limits_{C^k}{|x|^{-(\sum\limits_{i=1}^m\|\alpha_i\|_{L^\infty}+m\varepsilon)q(x)-n}dx}\nonumber
\\
&\geq&\int\limits_{2^{k-1}}^{2^k}\int\limits_{S^{n-1}}{r^{-(\sum\limits_{i=1}^m\|\alpha_i\|_{L^\infty}+m\varepsilon)q_+-1}d\sigma(x')}dr \gtrsim \eta_+.2^{-k(\sum\limits_{i=1}^m\|\alpha_i\|_{L^\infty}+m\varepsilon)q_+}.
\nonumber
\end{eqnarray}
Thus, by (\ref{maxminvar}), we have $\big\|g\chi_k\big\|_{L^{q(\cdot)}_{\omega}}\gtrsim \eta_+^{\frac{1}{q_+}}.2^{-k(\sum\limits_{i=1}^m\|\alpha_i\|_{L^\infty}+m\varepsilon)\frac{q_+}{q_-}}.$ This finishes the proof of the estimation (\ref{gvar}).

Now, we define 
\[
\vartheta^{**}(\varepsilon)=\dfrac{\eta_{+}^{\frac{1}{q_+}}2^{k_0(\alpha(0)-(\sum\limits_{i=1}^{m}\|\alpha_i\|_{L^\infty}+\varepsilon m)\frac{q_+}{q_-})}\prod\limits_{i=1}^m \Big(1-2^{p_i(\alpha_i(0)-(\|\alpha_i\|_{L^\infty}+\varepsilon)\frac{q_{i-}}{q_{i+}})}\Big)^{\frac{1}{p_i}}}{\prod\limits_{i=1}^m I_i(\varepsilon).\Big(1-2^{\alpha(0)-(\sum\limits_{i=1}^{m}\|\alpha_i\|_{L^\infty}+\varepsilon m)\frac{q_+}{q_-}}\Big)^{\frac{1}{p}}}.
\]
By (\ref{bdtfiVarHez})-(\ref{gvar}), we estimate 
\begin{eqnarray}\label{bdtHfVarHerz}
&&\big\|H_{\Phi,\vec A}(\vec f)\big\|_{{\mathop K\limits^.}_{q(\cdot),\omega}^{{\alpha(\cdot)}, p}}\nonumber
\\
&&\gtrsim\varepsilon^{-m\varepsilon}\vartheta^{**}.\Big(\int\limits_{U}{\dfrac{\Phi(t)}{|t|^n}\prod\limits_{i=1}^{m}\min\Big\{\big\|A_i^{-1}(t)\big\|^{\frac{n}{q_{i+}}+\gamma_i}, \big\|A_i^{-1}(t)\big\|^{\frac{n}{q_{i-}}+\gamma_i}\Big\}}\big\|A_i^{-1}(t)\big\|^{\|\alpha_i\|_{L^{\infty}}}\nonumber
\\
&&\,\,\,\,\,\,\,\,\,\,\,\,\,\,\,\,\,\,\,\,\,\,\,\,\,\,
\,\,\,\,\,\,\,\,\,\,\,\,\,\,\,\times\prod\limits_{i=1}^m\big\|A_i^{-1}(t)\big\|^{\varepsilon}\varepsilon^{m\varepsilon}dt\Big).\prod\limits_{i=1}^m \big\|f_i\big\|_{{\mathop K\limits^.}_{q_i(\cdot),\omega_i}^{{\alpha_i(\cdot)}, p_i}}.
\end{eqnarray}
Because of assuming $\alpha_i(0)<\|\alpha_i\|_{L^\infty}\frac{q_{i-}}{q_{i+}}$, we have $\alpha (0)<\sum\limits_{i=1}^m\big\|\alpha_i\big\|_{L^\infty}.$ From this, the limit of function $\varepsilon^{-m\varepsilon}\vartheta^{**}$ is a positive number when $\varepsilon$ tends to zero.
Therefore, by (\ref{prodLeb}), (\ref{bdtHfVarHerz}) and the dominated convergence theorem of  Lebesgue, we obtain
\[
\big\|H_{\Phi,\vec A}(\vec f)\big\|_{{\mathop K\limits^.}_{q(\cdot),\omega}^{{\alpha(\cdot)}, p}}
\gtrsim \mathcal C_6^*.\prod\limits_{i=1}^m \big\|f_i\big\|_{{\mathop K\limits^.}_{q_i(\cdot),\omega_i}^{{\alpha_i(\cdot)}, p_i}},
\]
which ends the proof for this case.
\end{proof}
When all of $\alpha_1(\cdot)$,...,$\,\alpha_m(\cdot)$ and $q_1(\cdot)$,...,$\,q_m(\cdot)$ are constant, we obtain the following useful result which is seen as an extension of Theorem 3.1 and Theorem 3.2 in the work \cite{CHH2016} to the case of matrices having property (\ref{DK1}) as mentioned above.
\begin{theorem}\label{TheoremMorreyHerz}
 Let $\omega(x)= |x|^\gamma$, $\gamma_1, ..., \gamma_m\in\mathbb R$, $\lambda_1,...,\lambda_m \in\mathbb R^+$, $\alpha_1,...,\alpha_m\in\mathbb R$, $1\leq q_i, q<\infty$, $0< p_i, p<\infty$ and $\omega_i(x)=|x|^{\gamma_i}$ for all $i= 1,...,m$. Simultaneously, let
\begin{equation}\label{gammq}
\frac{\gamma}{q}=\frac{\gamma_1}{q_1}+\cdots+\frac{\gamma_m}{q_m}.
\end{equation}
Then $H_{\Phi,\vec{A}}$ is a bounded operator from ${M\mathop{K}\limits^.}_{{p_1},{q_1}}^{{\alpha _1},{\lambda _1}}({\omega_1}) \times \cdots \times {M\mathop{K}\limits^.}_{{p_m},{q_m}}^{{\alpha _m},{\lambda _m}}({\omega_m})$ to ${M\mathop{K}\limits^.}_{{p},{q}}^{{\alpha},{\lambda}}({\omega})$ if and only if  
\[
\mathcal C_7=\int\limits_{{\mathbb R^n}} {\frac{{\Phi (t)}}{{{{\left| t \right|}^n}}}\prod\limits_{i = 1}^m {{{\left\| {A_i^{ - 1}(t)} \right\|}^{ - {\lambda _i} + {\alpha _i} + \frac{{n+{\gamma _i}}}{{{q_i}}}}}} } dt <  + \infty, 
\]
Moreover,
\[
{\big\| {{H_{\Phi ,\vec A }}} \big\|_{{M\mathop{K}\limits^.}_{{p_1},{q_1}}^{{\alpha _1},{\lambda _1}}({\omega_1}) \times \cdots \times {M\mathop{K}\limits^.}_{{p_m},{q_m}}^{{\alpha _m},{\lambda _m}}({\omega_m})\to {M\mathop{K}\limits^.}_{{p},{q}}^{{\alpha},{\lambda}}({\omega})}} \simeq \mathcal C_7.
\]
\end{theorem}
\begin{proof}
It is clear to see that the resuts of Theorem \ref{TheoremMorreyHerz} can be viewed as consequence of Theorem \ref{TheoremVarMorreyHerz1}. Indeed, we put $\gamma^*=\frac{\gamma}{q}$,$\gamma^*_i=\frac{\gamma_i}{q_i}$ for $i=1,...,m$ and $\omega^*=|x|^{\gamma^*}, \omega^*_i=|x|^{\gamma^*_i}$ for $i=1,...,m$. By having (\ref{gammq}) and assuming that $\alpha_1(\cdot)$,...,$\,\alpha_m(\cdot)$ and $q_1(\cdot)$,...,$\,q_m(\cdot)$ are constant, we have
\[
\mathcal C_5= \mathcal C_5^*= \int\limits_{{\mathbb R^n}} {\frac{{\Phi (t)}}{{{{\left| t \right|}^n}}}\prod\limits_{i = 1}^m {{{\left\| {A_i^{ - 1}(t)} \right\|}^{ - {\lambda _i} + {\alpha _i} + \frac{n}{q_i}+\gamma^*_i}}} }dt=\mathcal C_7.
\]
Thus, combining the case $\rm (a)$ and case $\rm (b1)$ of Theore \ref{TheoremVarMorreyHerz1}, we deduce
\[
{\big\| {{H_{\Phi ,\vec A }}} \big\|_{{M\mathop{K}\limits^.}_{{p_1},{q_1},\omega^*_1}^{{\alpha _1},{\lambda _1}}\times \cdots \times {M\mathop{K}\limits^.}_{{p_m},{q_m},{\omega^*_m}}^{{\alpha _m},{\lambda _m}}\to {M\mathop{K}\limits^.}_{{p},{q},{\omega}^*}^{{\alpha},{\lambda}}}} \simeq \mathcal C_7.
\]
At this point, by relation (\ref{rel-func}), we immediately get the desired result.
\end{proof}
As a consequence of Theorem \ref{TheoremVarHerz1}, we also obtain the analogous result for the
constant parameters case as follows.
\begin{theorem}\label{Herz}
Let $1\leq p, p_1,...,p_m <\infty$, the assumptions of Theorem \ref{TheoremMorreyHerz} and the hypothesis (\ref{lienhieppi}) in Theorem \ref{TheoremVarHerz} hold. We have that
 $H_{\Phi,\vec{A}}$ is a bounded operator from ${\mathop K\limits^.}_{{q_1}}^{ {\alpha _1},{p_1}}({\omega_1}) \times \cdots \times {\mathop K\limits^.}_{{q_m}}^{ {\alpha _m},{p_m}}({\omega_m})$ to ${\mathop K\limits^.}_q^{ \alpha ,p}(\omega)$ if and only if 
\[
\mathcal C_8= \int\limits_{{\mathbb R^n}} {\frac{{\Phi (t)}}{{{{\left| t \right|}^n}}}\prod\limits_{i = 1}^m {{{\left\| {A_i^{ - 1}(t)} \right\|}^{{\alpha _i} + \frac{{n+{\gamma _i}}}{{{q_i}}}}}} } dt <  + \infty.
\]
Furthermore,
\[
\big\| {{H_{\Phi ,\vec A }}} \big\|_{{\mathop K\limits^.}_{{q_1}}^{ {\alpha _1},{p_1}}({\omega_1}) \times \cdots \times {\mathop K\limits^.}_{{q_m}}^{ {\alpha _m},{p_m}}({\omega_m}) \to {\mathop K\limits^.}_q^{ \alpha ,p}(\omega)}\simeq \mathcal C_8.
\]
\end{theorem}
\begin{proof}
By putting $\gamma^*,\gamma^*_1,...,\gamma^*_m$ and $\omega^*,\omega^*_1,...,\omega^*_m$ above, it is not hard to see that $\mathcal C_6=\mathcal C_6^* =\mathcal C_8$. Therefore, by using case $\rm a$, case $\rm b1$ of Theorem \ref{TheoremVarHerz1} and the relation (\ref{rel-func}), we finish the proof of this theorem.
\end{proof}
Now, let us take measurable functions $s_1(t),..., s_m(t)\neq 0$ almost everywhere in $\mathbb R^n$. We consider a special case that the matrices $A_i(t)={\rm diag }[ s_{i1}(t),...,s_{in}(t) ]$ with $|s_{i1}|=\cdots=|s_{in}|=|s_i|$, for almost everywhere $t\in\mathbb R^n$, for all $i=1,...,m$. It is obvious that the matrices $A_i$'s satisfy the condition (\ref{DK1}). Therefore, since the Lebesgue space with power weights is a special case of the Herz space, we also obtain the following corollary.
\begin{corollary}\label{HquaHerz}
Let $1\leq p, p_1,...,p_m <\infty$, $\alpha_1,...,\alpha_m\in\mathbb R$, and the hypothesis (\ref{lienhieppi}) in Theorem \ref{TheoremVarHerz} is true. Then $H_{\Phi,\vec{A}}$ is a bounded operator from $L^{p_1}(|x|^{\alpha_1p_1}dx)\times \cdots \times L^{p_n}(|x|^{\alpha_np_n}dx)$ to $L^{p}(|x|^{\alpha p}dx)$ if and only if 
\[
\mathcal C_9= \int\limits_{{\mathbb R^n}} {\frac{{\Phi (t)}}{{{{\left| t \right|}^n}}}\prod\limits_{i = 1}^m {{{|s_i(t)|}^{{-\alpha _i} - \frac{n}{{p_i}}}}} } dt <  + \infty.
\]
Furthermore,
\[
\big\| {{H_{\Phi ,\vec A }}} \big\|_{L^{p_1}(|x|^{\alpha_1p_1}dx)\times \cdots \times L^{p_n}(|x|^{\alpha_np_n}dx) \to L^{p}(|x|^{\alpha p}dx)}= \mathcal C_9.
\]
\end{corollary}
\begin{proof}
 By the assumption of the matrices $A_i$'s, it is easy to see that
\[
|A_i(t)x|^{\alpha} = |s_i(t)|^\alpha.|x|^{\alpha},\,\textit{\rm for all}\,\, \alpha\in\mathbb R,\, i=1,...,m.
\]
Hence, we immedialately obtain the desired result.
\end{proof}
By the relation between the Hausdorff operators and the Hardy-Ces\`{a}ro operators as mentioned in Section 1, we see that Corollary \ref{HquaHerz} extends and strengthens the results of Theorem 3.1 in \cite{HK2015} with power weights.
\vskip 5pt
Let us now assume that $q(\cdot)$ and $q_i(\cdot)\in\mathcal P_\infty(\mathbb R^n)$, $\lambda,\alpha,\gamma ,\alpha_i, {\lambda _i},{\gamma_i}$ are real numbers such that $\lambda_i\in \big(\frac{-1}{q_{i\infty}},0\big)$, $\gamma_i \in (-n,\infty)$, $i=1,2,...,m$ and 
$$   \frac{1}{{{q_1}(\cdot)}} + \frac{1}{{{q_2}(\cdot)}} + \cdots + \frac{1}{{{q_m}(\cdot)}} = \frac{1}{q(\cdot)}, $$
  $$   \frac{{\gamma _1}}{{{q_{1\infty}}}} + \frac{\gamma _2}{{{q_{2\infty}}}} + \cdots + \frac{\gamma _m}{{{q_{m\infty}}}} = \frac{\gamma}{q_\infty},$$  
   $$  \frac{{n + {\gamma _1}}}{{n + \gamma}}{\lambda _1} + \frac{{n + {\gamma_2}}}{{n + \gamma }}{\lambda _2} +\cdots+ \frac{{n + {\gamma _m}}}{{n + \gamma}}{\lambda _m} = \lambda,$$  
   $$ \alpha_1+\cdots+\alpha_m=\alpha.$$
We are also interested in the multilinear Hausdorff operators on the product of weighted $\lambda$-central Morrey spaces with variable exponent. We have the following interesting result.
\begin{theorem}\label{TheoremMorreyVar}
Let $\omega_1(x)=|x|^{\gamma_1},...,\omega_m(x)=|x|^{\gamma_m}$, $\omega(x)= |x|^\gamma$ and $v_1(x)=|x|^{\alpha_1},...,v_m(x)=|x|^{\alpha_m}$, $v(x)= |x|^\alpha$. In addition, the hypothesis  (\ref{DKVarLeb}) in Theorem \ref{TheoremVarLebesgue1}
holds and  the following condition is true: 
\begin{equation}\label{DK2MVar}
\mathcal C_{10}= \int\limits_{\mathbb R^n} { \frac{{ {\Phi (t)}}}{{{{\left| t \right|}^n}}}\prod\limits_{i = 1}^m {{{\left\| {A_i(t)} \right\|}^{(n + {\gamma _i})\left(\frac{1}{q_{i\infty}}+{\lambda _i}\right)}}}}c_{A_i,q_i,\alpha_i}(t)\big\|1\big\|_{L^{r_{1i}(t,\cdot)}}dt <+ \infty,
\end{equation}
Then, we have $H_{\Phi ,\vec{A}}$ is bounded from ${\mathop B\limits^.}_{\omega_1,v_1}^{q_1(\cdot),\lambda _1}\times  \cdots\times {\mathop B\limits^.}_{{\omega_m,v_m}}^{{q_m(\cdot)},{\lambda _m}}$ to ${\mathop B\limits^.}_{\omega,v}^{q(\cdot),\lambda}$. 
\end{theorem}
\begin{proof}
For $R>0$, we denote
\[
 \Delta_R={\frac{1}{{{\omega}\big(B(0,R)\big)^{\frac{1}{q_\infty}+\lambda}}}} \big\|H_{\Phi ,\vec{A}}(\vec{f})\big\|_{L^{q(\cdot)}_{v}(B(0,R))}.
\]
It follows from using the Minkowski inequality for the variable Lebesgue space that
\begin{equation}\label{MVarDeltaR}
\Delta_R \lesssim \int\limits_{\mathbb R^n}{\frac{1}{{\omega(B(0,R))}^{\frac{1}{q_\infty} +\lambda}}.\frac{\Phi(t)}{|t|^n}\big\|\prod\limits_{i=1}^{m}{f_i(A_i(t).)}\big\|_{L^{q(\cdot)}_{v}(B(0,R))}dt}.
\end{equation}
On the other hand, we apply the H\"{o}lder inequality for the variable Lebesgue space to obtain
\begin{equation}\label{MVarHolder1}
\big\| \prod\limits_{i=1}^{m}f_i(A_i(t).)\big\|_{L^{q(\cdot)}_{v}(B(0,R))}\lesssim \prod\limits_{i=1}^{m}\big\|f_i(A_i(t).)\big\|_{L^{q_i(\cdot)}_{v_i}(B(0,R))}.
\end{equation}
By estimating as (\ref{MVarfiBLeb}) and (\ref{nhungVarLeb2}), we have
\begin{equation}\label{MVarfiB}
\big\|f_i(A_i(t).)\big\|_{L^{q_i(\cdot)}_{v_i}(B(0,R))}\lesssim c_{A_i,q_i,\alpha_i}(t).\big\|1 \big\|_{L^{r_{1i}(t,\cdot)}}.\big\|f_i\big\|_{L^{q_i(\cdot)}_{v_i}(B(0,R||A_i(t)||))}.
\end{equation}
In view of $\frac{{n + {\gamma _1}}}{{n + \gamma}}{\lambda _1} + \frac{{n + {\gamma_2}}}{{n + \gamma }}{\lambda _2} +\cdots+ \frac{{n + {\gamma _m}}}{{n + \gamma}}{\lambda _m} = \lambda$, we estimate
\[
\frac{1}{{\omega( B(0,R))}^{\frac{1}{q_{\infty}} + \lambda}} \lesssim \frac{\big\|A_i(t)\big\|^{(\gamma_i+n)(\frac{1}{q_{i\infty}}+\lambda_i)}}{{\omega_i( B(0,R\|A_i(t)\|))}^{\frac{1}{q_{i\infty}} + \lambda _i}}.
\]
Thus, by (\ref{MVarDeltaR}) and (\ref{MVarfiB}), it follows that
$\Delta_R\lesssim \mathcal C_{10}\prod\limits_{i = 1}^m {{{\left\| {{f_i}} \right\|}_{{\mathop B\limits^.}_{{\omega_i,v_i}}^{{q_i(\cdot)},{\lambda _i}}}}}.$ 
\\
Consequently, it is straightforward ${\left\| {{H_{\Phi ,\vec{A} }}(\vec{f} )} \right\|_{{\mathop B\limits^.}_{\omega,v}^{q(\cdot),\lambda }}}\lesssim \mathcal C_{10}\prod\limits_{i = 1}^m {{{\left\| {{f_i}} \right\|}_{{\mathop B\limits^.}_{{\omega_i,v_i}}^{{q_i(\cdot)},{\lambda _i}}}}}.$
\end{proof}
As a consequence of Theorem \ref{TheoremMorreyVar}, by the reason (\ref{detA}) and (\ref{DK1.1}), we also have the analogous result for the $q,q_1,...,q_m$-constant case as follows.
\begin{corollary}\label{CoroMorreyVar}
Let $\omega_i,v_i,\omega,v$ be as Theorem \ref{TheoremMorreyVar}. In addition, the following condition holds:
\begin{equation}\label{DK4MVarCoro}
\mathcal C_{11}= \int\limits_{\mathbb R^n} { \frac{{ {\Phi (t)}}}{{{{\left| t \right|}^n}}}\prod\limits_{i = 1}^m {{{\left\| {A_i^{ - 1}(t)} \right\|}^{\alpha_i-\frac{\gamma_i}{q_i} - \lambda_i(n + {\gamma _i})}}} } dt<+ \infty.
\end{equation}
Then, we have 
\[
{\left\| {{H_{\Phi ,\vec{A} }}(\vec{f} )} \right\|_{{\mathop B\limits^.}_{\omega,v}^{q,\lambda }}}\lesssim \mathcal C_{11}.\prod\limits_{i = 1}^m {{{\left\| {{f_i}} \right\|}_{{\mathop B\limits^.}_{{\omega_i,v_i}}^{{q_i},{\lambda_i}}}}}.
\]
\end{corollary}
\begin{proof}
By (\ref{detA}) and (\ref{DK1.1}), it is clear to see that
\[
{{\left\| {A_i(t)} \right\|}^{(n + {\gamma _i})\left(\frac{1}{q_{i\infty}}+{\lambda _i}\right)}} c_{A_i,q_i,\alpha_i}(t)\lesssim {{\left\| {A_i^{ - 1}(t)} \right\|}^{\alpha_i-\frac{\gamma_i}{q_i} - \lambda_i(n + {\gamma _i})}}.
\]
Hence, by Theorem (\ref{TheoremMorreyVar}), the proof is finished.
\end{proof}
Moreover, we also obtain the above operator norm on the product of weighted $\lambda$-central Morrey spaces as follows.
\begin{theorem}\label{TheoremMorrey}
Let $\omega(x)= |x|^\gamma$ and  $\omega_i(x)=|x|^{\gamma_i}$ for $i= 1,...,m$.
Then, we have that $H_{\Phi ,\vec{A}}$ is bounded from ${\mathop B\limits^.}^{q_1,\lambda _1}(\omega_1)\times\cdots \times {\mathop B\limits^.}^{{q_m},{\lambda _m}}(\omega_m)$ to ${\mathop B\limits^.}^{q,\lambda}(\omega)$ if and only if 
\[
\mathcal C_{12}= \int\limits_{\mathbb R^n} {\frac{\Phi (t)}{{\left| t \right|}^n}\prod\limits_{i = 1}^m {{{\left\| {A_i^{ - 1}(t)} \right\|}^{ - (n + {\gamma_i}){\lambda _i}}}} }dt <  + \infty.
\]
Furthermore, we obtain
\[
{\big\|{H_{\Phi ,\vec{A} }} \big\|_{{\mathop B\limits^.}^{{q_1},{\lambda_1}}({{\omega_1}}) \times  \cdots \times {\mathop B\limits^.}^{{q_m},{\lambda _m}}({{\omega_m}}) \to {\mathop B\limits^.}^{q,\lambda } (_\omega)}}\simeq \mathcal C_{12}.
\]
\end{theorem}
\begin{proof}
We first note that the sufficient condition of the theorem is derived from Corollary  \ref{CoroMorreyVar}. In more details, by letting $\alpha_i=\frac{\gamma_i}{q_i}$ for $i=1,...,m$, we have $\mathcal C_{11}=\mathcal C_{12}<\infty$. Hence, by Corollary \ref{CoroMorreyVar}, we find
\[
{\left\| {{H_{\Phi ,\vec{A} }}(\vec{f} )} \right\|_{{\mathop B\limits^.}_{\omega,{\omega}^{1/q}}^{q,\lambda }}}\lesssim \mathcal C_{12}.\prod\limits_{i = 1}^m {{{\left\| {{f_i}} \right\|}_{{\mathop B\limits^.}_{{\omega_i,{\omega_i}^{1/q_i}}}^{{q_i},{\lambda_i}}}}}.
\]
From this, by ${\mathop B\limits^.}_{\omega,{\omega}^{1/q}}^{q,\lambda }={\mathop B\limits^.}^{q,\lambda}(\omega)$ and ${\mathop B\limits^.}_{\omega_i,{\omega_i}^{1/q_i}}^{q_i,\lambda_i }={\mathop B\limits^.}^{q_i,\lambda_i}(\omega_i)$ for $i=1,...,m$, the proof of sufficient condition of this theorem is ended.

To give the proof for the necessary condition, let us now choose 
\[
f_i(x)=\left|x\right|^{(n+\gamma_i)\lambda_i}.
\] 
Then, it is not hard to show that
\[
{\left\| {{f_i}} \right\|_{{\mathop B\limits^.}^{{q_i},{\lambda _i}}({{\omega_i}})}} = {\left( {\frac{{n + {\gamma _i}}}{{\left| {{S_{n - 1}}} \right|}}} \right)^{{\lambda _i}}}\frac{1}{{{{\left( {1 + {q_i}{\lambda _i}} \right)}^{\frac{1}{{q_i}}}}}}.
\]
Thus, we have
\begin{equation}\label{Choosefi}
\prod\limits_{i = 1}^m {{{\big\| {{f_i}} \big\|}_{{\mathop B\limits^.}^{{q_i},{\lambda _i}}({\omega_i})}}}  \lesssim {\left( {\frac{{\gamma + n}}{{\left| {{S_{n - 1}}} \right|}}} \right)^{\lambda }}{(1 + \lambda q)^{\frac{-1}{q}}}.
\end{equation}
By choosing $f_i$'s, we also have
\begin{eqnarray}
  &&{\left\| {{H_{\Phi ,\vec{A} }}\left( {\vec{f} } \right)} \right\|_{{\mathop B\limits^.}^{q,\lambda }(\omega)}}
  \nonumber
  \\
   &&= \mathop {\sup }\limits_{R > 0} {\Big( {\frac{1}{{\omega{{(B(0,R))}^{1 + q\lambda }}}}\int\limits_{B(0,R)} {{{\Big| {\int\limits_{\mathbb R^n} {\frac{{\Phi (t)}}{{{{\left| t \right|}^n}}}} \prod\limits_{i = 1}^m {{{\left| {{A_i}(t)x} \right|}^{(n + {\gamma _i}){\lambda _i}}}} dt} \Big|}^q}\omega(x)dx} }\Big)^{\frac{1}{q}}}.
\nonumber
\end{eqnarray}
By (\ref{DK1.2}), we get ${\left| {{A_i}(t)x} \right|^{(n + {\gamma _i}){\lambda _i}}} \gtrsim {\left\| {A_i^{ - 1}(t)} \right\|^{ - (n + {\gamma_i}){\lambda _i}}}.{\left| x \right|^{(n + {\gamma_i}){\lambda _i}}}.$ Therefore, we imply that
\begin{eqnarray}
{\left\| {{H_{\Phi ,\vec{A} }}\big( {\vec{f} } \big)} \right\|_{{\mathop B\limits^.}^{q,\lambda }(\omega)}}&\gtrsim & \Big( {\int\limits_{{\mathbb R^n}} {\frac{{\Phi (t)}}{{{{\left| t \right|}^n}}}} {\prod\limits_{i=1}^{m}{\left\| {A_i^{ - 1}(t)} \right\|}^{ - (n + {\gamma_i}){\lambda _i}}}dt} \Big)\times\nonumber
\\
&&\,\,\,\times\,\mathop {\sup }\limits_{R > 0} {\Big( {\frac{1}{{\omega{{(B(0,R))}^{1 + q\lambda }}}}\int\limits_{B(0,R)} \Big({\prod\limits_{i = 1}^m {{{\left| x \right|}^{(n + {\gamma_i}){\lambda _i}q}\Big)|x|^{\gamma}}} dx} } \Big)^{\frac{1}{q}}}\nonumber
  \\
&=&\Big( {\int\limits_{{\mathbb R^n}} {\frac{{\Phi (t)}}{{{{\left| t \right|}^n}}}} {\prod\limits_{i=1}^{m}{\left\| {A_i^{- 1}(t)} \right\|}^{ - (n + {\gamma_i})\lambda_i}}dt} \Big){\Big( {\frac{{\gamma  + n}}{{\left| {{S_{n - 1}}} \right|}}} \Big)^{\lambda }}{(1 + \lambda q)^{\frac{-1}{q}}}. \nonumber
\end{eqnarray}
Hence, it follows from (\ref{Choosefi}) that
\[
{\left\| {{H_{\Phi ,\vec{A} }}\big( {\vec{f} } \big)} \right\|_{{\mathop B\limits^.}^{q,\lambda }(\omega)}} \gtrsim\Big( {\int\limits_{{\mathbb R^n}} {\frac{{\Phi (t)}}{{{{\left| t \right|}^n}}}} {\prod\limits_{i=1}^{m}{\left\| {A_i^{ - 1}(t)} \right\|}^{ - (n + {\gamma_i}){\lambda _i}}}dt} \Big).\prod\limits_{i = 1}^m {{{\left\| {{f_i}} \right\|}_{{\mathop B\limits^.}^{{q_i},{\lambda _i}}({{\omega_i}})}}}.
\]
Because of assuming that $H_{\Phi ,\vec{A}}$ is bounded from ${\mathop B\limits^.}^{q_1,\lambda _1}({\omega_1})\times \cdots \times {\mathop B\limits^.}^{{q_m},{\lambda _m}}({\omega_m})$ to ${\mathop B\limits^.}^{q,\lambda}(\omega)$, it immediately deduces that $\mathcal C_{12}<\infty$, and hence, the proof of the theorem is completed.
\end{proof}
{\textbf{Acknowledgments}}.  This paper is supported by the Vietnam National Foundation for Science and Technology Development (NAFOSTED) under grant number 101.02-2014.51.

\bibliographystyle{amsplain}

\begin{thebibliography}{79}

\bibitem{ALP} J. Alvarez, J. Lakey, M. Guzm\'{a}n-Partida, \textit{Spaces of bounded  $\lambda$-central mean oscillation, Morrey spaces, and $\lambda$-central Carleson measures}, Collect. Math. 51 (2000), 1-47.

\bibitem{Andersen} K. Andersen and E. Sawyer, \textit{Weighted norm inequalities for the Riemann-Liouville and Weyl fractional integral operators}, Trans. Amer. Math. Soc. 308 (1988), 547-558.

\bibitem{Almeida} A. Almeida,  and P. H\"{a}st\"{o}, \textit{Besov spaces with variable smoothness and integrability}, J. Funct. Anal. 258(5) (2010), 1628-1655.

\bibitem{Almeida2012} A. Almeida, D. Drihem, \textit{Maximal, potential and singular type operators on Herz spaces with variable exponents}, J. Math. Anal. Appl. 394 (2012), 781-795.

\bibitem{BM} G. Brown,  F. M\'{o}ricz, \textit{Multivariate Hausdorff operators on the spaces $L^p(\mathbb R^n)$}, J. Math. Anal. Appl. 271 (2002), 443-454.

\bibitem{Bandaliev2010} R.A. Bandaliev, \textit{The boundedness of multidimensional hardy operators in weighted variable Lebesgue spaces}, Lith. Math. J. 50 (2010), 249-259.

\bibitem{Capone} C. Capone, D. Cruz-Uribe, and A. Fiorenza, \textit{The fractional maximal operator and fractional integrals on variable $L_p$ spaces}, Rev. Mat. Iberoam. 23 (3) (2007), 743-770. 

\bibitem{Carton-Lebrun} C. Carton-Lebrun and M. Fosset, \textit{Moyennes et quotients de Taylor dans BMO}, Bull. Soc. Roy. Sci. Li\'{e}ge 53, No. 2 (1984), 85-87.

\bibitem{Coifman1} R. R. Coifman, Y. Meyer, \textit{On commutators of singular integrals and bilinear singular integrals}, Trans. Amer. Math. Soc. 212 (1975), 315-331.

\bibitem{Coifman2} R. R. Coifman, Y. Meyer, \textit{Au del\`{a} des op\'{e}rateurs pseudo-diff\'{e}rentiels}, Ast\'{e}risque, 57 (1978).

\bibitem{Chuong} N. M. Chuong, \textit{Degenerate parabolic pseudodifferential operators of variable order}, Dokl. Akad. Nauk SSSR 268 (1983), 1055-1058.

\bibitem{CH2014} N. M. Chuong, H. D. Hung, \textit{Bounds of weighted Hardy-Ces\`{a}ro operators on weighted Lebesgue and BMO spaces}, Integr. Transforms and Special Funct. 25 (2014), 697-710.

\bibitem{CDH2016} N. M. Chuong, D. V. Duong, H. D. Hung, \textit{Bounds for the weighted Hardy-Ces\`{a}ro operator and its commutator on weighted Morrey-Herz type spaces}, Z. Anal. Anwend. 35 (2016) 489-504.

\bibitem{CHH2016} N. M. Chuong, N. T. Hong, H. D. Hung, \textit{Multilinear Hardy-Cesaro operator and commutator on the product of Morrey-Herz spaces}, Analysis Math. (To appear).

\bibitem{Chuong2016} N. M. Chuong, \textit{Pseudodifferential operators and wavelets over real and p-adic fields}, Springer (Submitted to the Editor of Springer).

\bibitem{Christ} M. Christ and L. Grafakos, \textit{Best constants for two non-convolution inequalities}, Proc. Amer. Math. Soc. 123 (1995), 1687-1693.

\bibitem{Cruz-Uribe} D. Cruz-Uribe, A. Fiorenza, J. M. Martell, and C. Perez, \textit{The boundedness of classical operators on variable $L_p$ spaces}, Ann. Acad. Sci. Fenn. Math. 31(1) (2006), 239-264.

\bibitem{CUF2013} D. Cruz-Uribe, A. Fiorenza, \textit{Variable Lebesgue Spaces: Foundations and Harmonic Analysis}, Springer-Basel, 2013.

\bibitem{Diening} L. Diening,  M. R{u}˚\v{z}i\v{c}ka, \textit{Calder\'{o}n-Zygmund operators on generalized Lebesgue spaces $L^{p(x)}$ and problems related to ﬂuid dynamics}, J. Reine Angew. Math. 563 (2003), 197-220. 

\bibitem{Diening1} L. Diening, P. Harjulehto, P. H\"{a}st\"{o}, M. Ruzicka, \textit{Lebesgue and Sobolev spaces with variable exponents}, Springer-Verlag, (2011).

\bibitem{FGLY2015} Z. W. Fu, S. L. Gong, S. Z. Lu and W. Yuan, \textit{Weighted multilinear Hardy operators and commutators}, Forum Math. 27 (2015), 2825-2851.

 \bibitem{Guliyev} V. S. Guliyev, J. Hasanov, and S. Samko, \textit{Boundedness of the maximal, potential and singular operators in the generalized variable exponent Morrey spaces},  Math. Scand., 107 (2010), 285-304.

\bibitem{Ge} C. Georgakis, \textit{The Hausdorff mean of a Fourier-Stieltjes transform}, Proc. Amer. Math. Soc. 116 (1992), 465 - 471.
 
\bibitem{Hausdorff} F. Hausdorff, \textit{Summation methoden und Momentfolgen}, I, Math. Z. 9 (1921), 74-109.

\bibitem{Hurwitz} W. A. Hurwitz, L. L. Silverman, \textit{The consistency and equivalence of certain definitions of summabilities}, Trans. Amer.  Math. Soc. 18 (1917), 1-20.

\bibitem{HK2015} H. D. Hung, L. D. Ky, \textit{New weighted multilinear operators and commutators of Hardy-Ces\`{a}ro type}, Acta Math. Sci. Ser. B Engl. Ed. 35 (2015)(6), 1411-1425.

\bibitem{H2000} W. Hoh, \textit{Pseudodifferential operators with negative definite symbols of varable order}, Revista Mat. Iberoamer. 18, No.2 (2000), 219-241.

\bibitem{Kovacik} O. Kov\'{a}\v{c}ik,  J. R\'{a}kosn\'{i}k, \textit{On spaces $L^{p(x)}$ and $W^{k, p(x)}$}, Czechoslovak Math. J. 41(116) (1991), 592-618.

\bibitem{Lerner} A. Lerner and E. Liflyand, \textit{ Multidimensional Hausdorff operators on real Hardy spaces}, J. Austr. Math. Soc. 83 (2007), 79-86.

\bibitem{Lu} S. Z. Lu, D. C. Yang, G. E. Hu, \textit{Herz type spaces and their applications}, Beijing Sci. Press (2008).

\bibitem{LZ2014} Y. Lu, Y. P. Zhu, \textit{Boundedness of multilinear Calder\'{o}n-Zygmund singular operators on Morrey-Herz spaces with variable exponents}, Acta Math. Sin.(Engl. Ser.) 30 (2014), 1180-1194.

\bibitem{Mamedov} F. I. Mamedov, A. Harman, \textit{On a Hardy type general weighted inequality in spaces $L^{p(\cdot)}$}, Integr. Equations Oper. Theor.  66 (2010), 565-592.

\bibitem{Mashiyev} R. Mashiyev,  B. \c{C}eki\c{c}, F. I. Mamedov, S. Ogras,  \textit{Hardy's inequality in power-type weighted $L^{p(\cdot)}(0,\infty)$}, J. Math. Anal. Appl. 334(1) (2007), 289-298.

\bibitem{Miyachi} A. Miyachi, \textit{ Boundedness of the Ces\`{a}ro operator in Hardy space}, J. Fourier Anal. Appl. 10 (2004), 83-92.

\bibitem{Moricz2005} F. M\'{o}ricz, \textit{Multivariate Hausdorff operators on the spaces $H^1(\mathbb R^n)$ and $BMO(\mathbb R^n)$}, Analysis Math. 31 (2005), 31-41.

\bibitem{Jacob} N. Jacob, H. G. Leopold, \textit{Pseudodifferential operators with variable order of differentiation generating Feller semigroups}, Integr. Equations Oper. Theor. 17 (1993), 544-553.

\bibitem{Rafeiro} H. Rafeiro, S. Samko, \textit{Hardy type inequality in variable Lebesgue spaces}, Ann. Acad. Sci. Fenn., Ser. A 1 Math. 34(1) (2009), 279-289 . 

\bibitem{WZ2016} J. L. Wu, W. J. Zhao, \textit{Boundedness for fractional Hardy-type operator on variable-exponent Herz-Morrey spaces}, Kyoto J. Math. Vol. 56, No. 4 (2016), 831-845.

\bibitem{XW2015} X. Wu, \textit{Necessary and sufficient conditions for generalized Hausdorff operators and commutators}, Ann. Funct. Anal. Vol. 6, No. 3 (2015), 60-72.

\bibitem {Stein} Elias M. Stein, \textit{Harmonic analysis: real-variable methods, orthogonality, and oscillatory integrals,} Princeton University Press, (1993).

\bibitem{Xiao2001} J. Xiao, \textit{$L^p$ and $BMO$ bounds of weighted Hardy-Littlewood averages}, J. Math. Anal. Appl. 262 (2001), 660-666.

\end{thebibliography}

\end{document}